\title{\bf Uniqueness of Banach space valued graphons}
\author{\sc D\'avid Kunszenti-Kov\'acs\\
\rm MTA Alfr\'ed R\'enyi Institute of Mathematics, Budapest,
Hungary\footnote{The research leading to these results has received funding from the European Research Council under the European Union's Seventh Framework Programme (FP7/2007-2013) / ERC grant agreement $\mr{n}^\circ$617747, and from the MTA R\'enyi Institute Lend\"ulet Limits of Structures Research Group.}
\\[1cm]}

\documentclass[11pt,a4paper]{article}

 \usepackage{hyperref}
\usepackage{enumerate,color, tensor}
\usepackage{mdwlist}
\usepackage{mathrsfs}
\usepackage{url}
\usepackage[arrow, matrix, curve]{xy}
\usepackage[american]{babel}
\usepackage{paralist}               	
\usepackage[parfill]{parskip}			
\usepackage{a4wide}
\usepackage{amsmath,amsthm}
\usepackage{amssymb}
\usepackage[T1]{fontenc}
\usepackage[bottom]{footmisc}
\usepackage{dsfont}

\newcommand{\mc}{\mathcal}
\newcommand{\ms}{\mathscr}

\newcommand{\mf}{\mathbf}
\newcommand{\mb}{\mathbb}
\newcommand{\mr}{\mathrm}

\def \ph {\varphi}
\def \lin {\operatorname{lin}}

\theoremstyle{plain}
\newtheorem*{thm*}{Theorem}
\newtheorem{thm}{Theorem}
\newtheorem{lemma}[thm]{Lemma}
\newtheorem{prop}[thm]{Proposition}
\newtheorem{cor}[thm]{Corollary}
\theoremstyle{definition}
\newtheorem{definition}[thm]{Definition}

\newtheorem{remark}[thm]{Remark}

\date{\today\\
\small Mathematics Subject Classification: 46G10, and 28B05, 05C60}

\begin{document}		

\maketitle

\begin{abstract}
A Banach space valued graphon is a function $W:(\Omega, \mc{A},\pi)^2\to\mc{Z}$ from a probability space to a Banach space with a separable predual, measurable in a suitable sense, and lying in appropriate $L^p$-spaces. 
As such we may consider $W(x,y)$ as a two-variable random element of the Banach space.
A two-dimensional analogue of moments can be defined with the help of graphs and weak-* evaluations, and a natural question that then arises is whether these generalized moments determine the function $W$ uniquely -- up to measure preserving transformations.

\noindent The main motivation comes from the theory of multigraph limits, where these graphons arise as the natural limit objects for convergence in a generalized homomorphism sense.\\
\noindent Our main result is that this holds true under some Carleman-type condition, but fails in general even with $\mc{Z}=\mb{R}$, for reasons related to the classical moment-problem. In particular, limits of multigraph sequences are uniquely determined - up to measure preserving transformations - whenever the tails of the edge-distributions stay small enough.

\end{abstract}

\maketitle

\section{Introduction}

Moment determinacy deals with the question of whether a given type of probability measure is uniquely determined by its moments. In the classical settings, the theory is rich and well understood. For instance, if the probability measure lives on a bounded interval (Hausdorff problem), then knowledge of the moments is enough to recover the measure. The same holds true for the vector-valued version, where the measure is supported in a bounded domain of $\mb{R}^k$ for some finite $k$. The notion of moments has to be slightly adapted though: to guarantee uniqueness, we need \emph{mixed moments}, i.e., the expectations of $\prod_{j=1}^k X_j^{\alpha_j}$, where $X_j$ is the $j$-th coordinate of the random vector ($1\leq j\leq k$), and the $\alpha_j$-s are nonnegative integers.\\
However, if the support is unbounded (cf. Stieltjes and Hamburger problems), there is no general positive or negative answer to moment determinacy, and additional conditions (Carleman, Krein, etc.) are needed to prove or disprove uniqueness, see e.g. the monograph \cite{Akh} by Akhiezer.

In a somewhat more general setting, given a measurable function $f:[0,1]\rightarrow\mb{R}$, we have an induced probability measure on $\mb{R}$. Knowing the moments of this probability measure, we may wish to recover $f$ itself. Clearly this is not possible, but if the measure is moment determinate, we may still recover a great deal of information about $f$. For instance, we may find a canonical monotone increasing representation $f':[0,1]\rightarrow\mb{R}$, where $f'$ is determined up to a null-set. Also, we may find measure preserving transformations of $[0,1]$ that transform $f$ and $f'$ into a.e. equal functions. A similar result holds for $f$ and $f'$ taking its values in $\mb{R}^k$, where the monotone reordering is extended to the lexicographic ordering.

In a recent paper by Borgs, Chayes and Lov\'asz (\cite{BCL}), the authors considered a variant of the above Hausdorff question, involving an extra dimension in the domain of $f$. Namely, they investigated bounded symmetric measurable functions $f:[0,1]^2\to \mb{R}$. Such two-variable functions also induce a probability measure on $[0,1]$, but the structure of the domain means that there is an added spatial correlation in the function values. They proved that with an appropriate notion of generalized moments that are adapted to the extra spatial dimension, all such functions are uniquely determined by their generalized moments, up to measure preserving transformation of the variables. Note however, that there unfortunately is no canonical reordering of the interval that would yield a "monotone" function in this two-variable setting.\\
The motivation for studying such functions comes from the theory of limits of simple dense graphs, developed by Borgs, Chayes, Lov\'asz, S\'os and Vesztergombi
\cite{BCLSV1,BCLSV2} and Lov\'asz and Szegedy \cite{LSz1}. Symmetric measurable functions $W:[0,1]^2\to [0,1]$ represent limit objects for graph sequences under a combinatorial/probabilistic notion of convergence connected to homomorphism densities. The deep connection between analysis and limit theories of combinatorial objects is further highlighted in the paper \cite{LSz2} by Lov\'asz and Szegedy, where Szemer\'edi's Regularity Lemma is reformulated and given analytic interpretations.
We also note that the question of uniqueness in the limit theory for hypergraphs was treated by Elek and Szegedy in \cite{ESz}, but their methods were of a fundamentally different nature, making use of ultraproducts.\\
Further developments in this field have led to the investigation of limits of multigraphs with no bound on the number of multiple edges between nodes, and more generally to limits of decorated graphs. For the multigraph setting, this at a first glance simply corresponds to passing from bounded to unbounded functions.
But it turns out that for combinatorial reasons one expects the limit functions to take measures as their values rather than simply a real number. In \cite{KKLSz}, Lov\'asz, Szegedy and the author developed a general functional analytic framework that allows one to handle the various possible combinatorial interpretations opened up by the multigraph/decorated graph setting and compare the corresponding convergence notions.

The limit objects/graphons this generalized setting leads to are symmetric, weak-* measurable functions $W:[0,1]^2\to\mc{Z}$, where $\mc{Z}$ is a Banach space with separable predual (typically a space of measures that depends on the specific combinatorial interpretation(s) studied). As such they are two-variable random elements of said Banach space. The \emph{homomorphism densities} that are used to define convergence are integrals of products of weak-* evaluations of the graphon (cf. Definition \ref{def:density}). These will be the two-dimensional generalizations of moments that are adapted to the added structure of the domain of the graphons.
For technical reasons, we shall not restrict ourselves to graphons $W$ with domain $[0,1]^2$, but rather more generally consider domains that are the product of a probability measure space with itself, $(\Omega,\mc{A},\pi)^2$.

After one defines a class of objects that is rich enough to capture the whole limit theory, an important question still remains, namely whether the class is too big or not. This was partly answered in \cite{KKLSz}, where it was shown that every graphon is a limit of a sequence of decorated graphs.
Our aim is to address the remaining part of the question: how much redundancy is there in the space of graphons? In other words, given two graphons, under what conditions do they represent the same limit?\\
The paper follows the approach of \cite{BCL}, building on and refining its ideas and proofs and combining them with functional analytic methods related to weak-* integrable functions to extend the results to a much more general setting that also includes limits of multigraphs with unbounded edge multiplicities. We show that if the generalized moments of a graphon satisfy a Carleman-type condition, then, similarly to the one-variable case, the graphon is uniquely determined, up to measure preserving transformations of the underlying space $(\Omega,\mc{A},\pi)$, see (Theorem \ref{thm:main}). In particular, bounded Banach space valued graphons are always moment determinate.\\
However, as for the classical moment problems, one may not forgo some type of bounds on the moments to guarantee uniqueness, and using moment indeterminate measures on $\mb{N}$, one can construct graphons the are not isomorphic in any sense, but have identical generalized moments (cf. Section \ref{Section:ex}).

\section{Preliminaries}

As mentioned in the introduction, we are interested in functions in two variables taking values in a Banach space $\mc{Z}$, and a corresponding notion of moments. This involves taking integrals, but there is no unique "natural" integral notion in Banach spaces.

 The Bochner, or strong integral, corresponds to integrability "in norm", and is defined as a limit of integrals of simple functions. The Pettis integral, or weak integral, uses duality to reduce integrability of a $\mc{Z}$-valued function to that of real valued ones through weak evaluations. Finally, if $\mc{Z}$ possesses a predual, weak-* integrability can also be defined in a similar way, as shall be done below. Of the three integral notions, weak-* integrability is the weakest property (provided it exists), and Bochner integrability the strongest, though some or all properties coincide under certain conditions on the Banach space $\mc{Z}$ (e.g. separability).

 In this paper we are interested in the largest class, that of weak-* integrable functions, as these are the ones that arise naturally as limits of multigraph sequences (see \cite{KKLSz}).
After introducing this class of functions and some of its properties, we shall turn our attention to the combinatorial structures that allow us to define a notion of moments adapted to the two-variable setting and the added geometric structure that comes with it.

\subsection{Weak-* integrable functions}

Let $\Phi$ be a separable Banach space, and let $\mc{Z}$ denote its dual. The elements of $\Phi$ act on $\mc{Z}$ as bounded linear functionals in the canonical way. Let further $\Psi\subset \Phi$ be a countable dense subset.

\begin{definition}
 Let $(\Omega, \mc{A},\pi)$ be a probability space. A function $W:(\Omega, \mc{A},\pi)\to\mc{Z}$ is called \emph{weak-* measurable} if for any $\varphi\in\Phi$, the function $\langle \varphi,W\rangle$ is measurable.\\
The weak-* measurable function $W$ is called \emph{weak-* scalarly integrable}, if for any $\varphi\in\Phi$, the function $\langle \varphi,W\rangle$ lies in $L^1(\pi)$.\\
The weak-* measurable function $W$ is called \emph{weak-* integrable} (or \emph{Gelfand integrable}), if there exists a mapping $\mu_W:\mc{A}\to \mc{Z}$ such that for any $A\in\mc{A}$ and $\ph\in\Phi$ we have
\[
\int_A \langle \ph,W\rangle=\ph(\mu(A)).
\]
\end{definition}

\begin{remark}
Note that the standard definition of weak-* measurability only requires that the weak-* evaluations be measurable with respect to the completion of the measure $\pi$. In this paper, however, we shall need to differentiate between functions that are measurable with respect to a measure, and those that are measurable only with respect to its completion.
\end{remark}

Clearly a weak-* integrable function is also weak-* scalarly integrable. The following classical result shows that the converse is also true. We include its proof for the readers' convenience.

\begin{prop}\label{prop:Gelfand}
Each weak-* scalarly integrable function $W:(\Omega, \mc{A},\pi)\to\mc{Z}$ is weak-* integrable.
\end{prop}
\begin{proof}
 For a given $A\in\mc{A}$, consider the linear map $\ms{W}_A:\Phi\to L^1(\Omega,\mc{A},\pi)$ given by $\ms{W}_A(\ph):=\mathds{1}_A\cdot\langle\ph,W\rangle$. Note that this map has a closed graph.\\
 Indeed, given any convergent sequence $\ph_n\to\ph\in\Phi$ with $\ms{W}_A(\ph_n)\to f\in L^1(\Omega,\mc{A},\pi)$, we can find a subsequence such that $\ms{W}_A(\ph_{n_k})$ converges almost everywhere to $f$. But pointwise $\ms{W}_A(\ph_n)$ converges to $\ms{W}_A(\ph)$, and so $f=\ms{W}_A(\ph)$ in $L^1(\Omega,\mc{A},\pi)$.\\
Thus by the Closed Graph Theorem $\ms{W}_A$ is bounded, and so
\[
\left|\int_A\langle\ph,W\rangle\,d\pi \right|\leq \int_A \left|\langle\ph,W\rangle\right|\,d\pi=\|\ms{W}_A(\ph)\|\leq
\|\ms{W}_A\|\cdot \|\ph\|,
\]
which means that the map $\ph\mapsto\int_A\langle\ph,W\rangle\,d\pi$ is a continuous linear functional on $\Phi$. Therefore there exists a representing element $\mu_W(A)\in\mc{Z}$, completing the proof.
\end{proof}

The next result shows that the existence of a Radon-Nikodym derivative extends to the setting of weak-* integrals, and it is a variant of a theorem due to Rybakov (\cite[Thm. 2]{Rybakov}). In his paper Rybakov assumes the underlying measure space to be complete, as the proof makes use of lifting on $L^\infty(\Omega,\mc{A})$. In our setting, however, the predual $\Phi$ is separable, and therefore lifting can be avoided through a different approach, and the assertion holds even when the measure space is not complete.

\begin{prop}\label{prop:RN}
Let $(\Omega,\mc{A},\pi)$ be a probability space, and suppose that the vector-valued measure $\mu:\mc{A}\to\mc{Z}$ is of $\sigma$-finite variation, and $\mu\ll\pi$. Then there exists a weak-* integrable function $W:\Omega\to\mc{Z}$ such that
\[
\langle \ph, \mu(A)\rangle =\int_A \langle \ph,W\rangle\,d\pi
\]
for every $\ph\in\Phi$ and $A\in\mc{A}$.
\end{prop}

\begin{proof}
First, assume that there exists a $c>0$ such that $\|\nu(A)\| \leq c\mu(A)$, for every $A\in\mc{A}$.
According to the classical Radon-Nikodym theorem for each
$\ph\in\Phi$ there exists a function $w'_\ph\in L^1(\Omega,\mc{A},\pi)$ such that for all $A\in\mc{A}$ we have

\begin{equation*}\label{eqn:RN'}
\langle \ph, \mu(A)\rangle = \int_A w'_\ph \,d\pi
\end{equation*}

Clearly $|w'_\ph| \leq c \|\ph\|$ a.e., for each $\ph$ separately. In the general case, this is the point where a lifting on $L^\infty(\Omega,\mc{A},\pi)$ would be used to assemble all these derivatives into a single $\mc{Z}$-valued function.

Instead, now let $\Psi':=\lin_{\mb{Q}}\Psi$, and consider an enumeration $\psi_0,\psi_1,\psi_2,\ldots$ of the countable set $\Psi'$, where $\psi_0=0$. Let us recursively do the following. Let $w_{\psi_0}\equiv 0$ and $w_{\psi_1}:=w'_{\psi_1}$, and for each $n\geq 2$, if $\psi_n$ is not in the linear hull of the previous $\psi_i$-s, then let $w_{\psi_n}:=w'_{\psi_n}$. If on the other hand $\psi_n=\sum_{i=1}^{n-1}a_i\psi_i$ for some $a_1,a_2,\ldots,a_{n-1}\in\mb{R}$, then let $w_{\psi_n}:=\sum_{i=1}^{n-1}a_iw_{\psi_i}$. Since Radon-Nikodym derivation is linear, this is well-defined, and
\begin{equation*}\label{eqn:RN}
\langle \psi_n, \mu(A)\rangle = \int_A w_{\psi_n} \,d\pi
\end{equation*}
for all $A\in\mc{A}$ and $1\leq n$. Also, this can be linearly extended to all $\psi\in\lin_{\mb{R}}\Psi=:\Psi''$.\\
Since $\Psi'\subset\Phi$ is countable, there exists a set $N\in\mc{A}$ with $\pi(N)=0$ such that $|w_\psi| \leq c \|\psi\|$ for all $\omega\in\Omega\backslash N$ and $\psi\in\Psi'$, and then by construction actually for all $\psi\in\Psi''$.
For each $\omega\in\Omega\backslash N$ define the functional $W'_\omega:\Psi''\to\mb{R}$ by
\[
W'_\omega(\psi):=w_\psi(\omega).
\]
Clearly $W'_\omega$ is linear and bounded by $c$ on $\Psi''$, and hence has a unique bounded extension $W_\omega\in\mc{Z}$ to $\Phi$.
Now we can define a function $W:(\Omega,\mc{A},\pi)\to\mc{Z}$ through
\begin{equation*}
W(\omega):=\left\{
\begin{array}{lr}
W_\omega & \mbox{ if } \omega\in\Omega\backslash N;\\
0 & \mbox{otherwise}.
\end{array}
\right.
\end{equation*}
Notice that $W$ is automatically $\mc{A}$-measurable, and we have $\|W(\omega)\|\leq c$ for all $\omega\in\Omega$. Using again that $\overline{\Psi''}=\Phi$, we may thus conclude that
\[
\langle \ph, \mu(A)\rangle =\int_A \langle \ph,W\rangle\,d\pi
\]
for every $\ph\in\Phi$ and $A\in\mc{A}$.\\
To finish the proof in the general case, we need that $\Omega$ can be written as the union of a set of measure zero, and sets $A_n\in\mc{A}$ ($n\in\mb{N}^+$) such that for each $n$, we have $\|\mu(A)\|\leq n\pi(A)$ for all $A_n\supset A\in\mc{A}$.
Since $\mu$ is of $\sigma$-finite variation, the Radon-Nikodym theorem still applies to the measures $\ph\circ\mu$, $\ph\in\Phi$. Define the functions $w_\psi$ ($\psi\in\Psi'$) as above.
For $n\in\mb{N}^+$, let $A_n:=\left\{\omega\in\Omega\Bigl|\left|w_\psi(\omega)\right|\leq n\|\psi\|\,\forall \psi\in\Psi'\Bigr.\right\}$. Then, since $\Psi'$ is countable and generates the norm on $\mc{Z}$, each $A_n$ is measurable and satisfies $\|\mu(A)\|\leq n\pi(A)$ for all $A_n\supset A\in\mc{A}$. It thus only remains to be shown that $\Omega\backslash\cup_{n\in\mb{N}^+}A_n=:S$ is a set of measure zero.

Since $\mu$ is of $\sigma$-finite variation, we can write $S=\bigcup_{i=1}^\infty S_i$ with $S_i\in\mc{A}$, and $\mu$ being of finite variation on each $S_i$. Suppose that there exists an $S_k$ with $d:=\pi(S_k)>0$, and let then $s$ be the variation of $\mu$ on $S_k$, and $q:=\left\lfloor\frac{s}{d}\right\rfloor+1$. For $j\geq 1$, define the sets $S^+_{k,j}\in\mc{A}$ and $S^-_{k,j}\in\mc{A}$ through
\[
S^+_{k,j}:=\left\{\omega\in S_k\Bigl|
\left|w_{\psi_a}(\omega)\right|< q\|\psi_a\|\,\forall\, 1\leq a<j, \mbox{ and } w_{\psi_j}(\omega)\geq q\|\psi_j\|
\Bigr.\right\},
\]
and
\[
S^-_{k,j}:=\left\{\omega\in S_k\Bigl|
\left|w_{\psi_a}(\omega)\right|< q\|\psi_a\|\,\forall\, 1\leq a<j, \mbox{ and } w_{\psi_j}(\omega)\leq -q\|\psi_j\|
\Bigr.\right\}.
\]
Then $\|\mu(S^\pm_{k,j})\|\geq q \pi(S^\pm_{k,j})$, and the sets $(S^\pm_{k,j})_{j\in\mb{N}^+}$ form a partition of $S_k$, hence
\[
s\geq \sum_{j=1}^\infty \|\mu(S^+_{k,j})\|+\sum_{j=1}^\infty \|\mu(S^-_{k,j})\|\geq q \sum_{j=1}^\infty\pi(S^+_{k,j})+\pi(S^-_{k,j})=q\pi(S_k)=qd>s,
\]
leading to a contradiction. Thus $S$ is the countable union of sets of measure zero, and the proof is complete.

\end{proof}

The following lemma is an easy corollary.

\begin{lemma}\label{le:conditional}
Let $(\Omega,\mc{A},\pi)$ be a probability space, and $W:(\Omega,\mc{A},\pi)\to\mc{Z}$ a weak-* integrable function with $\|W\|\in L^1(\pi)$. Let further $\mc{A}'\subset\mc{A}$ be a sub-$\sigma$-algebra. Then there exists a weak-* integrable function $W':(\Omega,\mc{A}',\pi|_{\mc{A}'})\to\mc{Z}$ such that
\[
\langle \ph, \mu(A')\rangle =\int_{A'} \langle \ph,W\rangle\,d\pi
\]
for every $\ph\in\Phi$ and $A'\in\mc{A}'$. If $W'_1$ and $W'_2$ are two such functions, then $W_1'=W'_2$ almost everywhere with respect to $\pi|_{\mc{A}'}$.\\
If for some $1\leq p<\infty$ we have $\|W\|\in L^p$, then also $\|W'\|\in L^p$ with $\|W\|_p\geq\|W'\|_p$.
\end{lemma}
\begin{proof}
Let the vector-valued measure $\mu:\mc{A}'\to\mc{Z}$ be defined through the weak-* integral $\mu(A'):=\int_{A'} W$. Then $\mu$ is clearly absolutely continuous with respect to $\pi|_{\mc{A}'}$ and has $\sigma$-finite variation. Hence by Proposition \ref{prop:RN}, there exists a weak-* integrable function $W':(\Omega,\mc{A}',\pi|_{\mc{A}'})\to\mc{Z}$ satisfying the required equality. If $W'_1$ and $W'_2$ are two such functions, then for each $\psi\in\Psi$ we have 
\[
\langle \psi,W'_1\rangle=E(\langle \psi,W\rangle\left|\mc{A}'\right.)=\langle \psi,W'_2\rangle
\]
$\pi|_{\mc{A}'}$-almost everywhere. Since $\Psi$ is countable and separates $\mc{Z}$, the assertion follows.\\
The inequality between the norms follows from the fact that $\left|\langle \ph,W(\cdot)\rangle\right|\leq\|W(\cdot)\|\cdot\|\ph\|$, and Radon-Nikodym derivation is order-preserving, hence a contraction on every $L^p$ space.
\end{proof}

This allows us to extend the notion of conditional expectation to weak-* integrable functions with values in $\mc{Z}$.

\begin{definition}
Let $(\Omega,\mc{A},\pi)$ be a probability space, and $W:(\Omega,\mc{A},\pi)\to\mc{Z}$ a weak-* integrable function. Let further $\mc{A}'\subset\mc{A}$ be a sub-$\sigma$-algebra. Then the $\pi|_{\mc{A}'}$-almost everywhere unique function $W'$ given in Lemma \ref{le:conditional} is called the \emph{conditional expectation} of $W$ with respect to the $\sigma$-algebra $\mc{A}'$, and will be denoted by $E(W\left|\mc{A}'\right.)$.\\
The function $W$ is said to be \emph{almost} $\mc{A}'$-\emph{measurable}, if $W=E(W\left|\mc{A}'\right.)$ holds $\pi$-almost everywhere.
\end{definition}

\begin{definition}
A symmetric  function $W:(\Omega, \mc{A},\pi)^2\to\mc{Z}$ is called
a $\mc{Z}$-\emph{graphon} if
it is weak-* measurable
with respect to the completion $\overline{\mc{A}\times\mc{A}}$ of the underlying $\sigma$-algebra,
and the function $(x,y)\mapsto \|W(x,y)\|_\mc{Z}$
lies in $\mc{L}_{(\Omega, \mc{A},\pi)}:=\bigcap_{1\leq p<\infty} L^p\left(\overline{(\Omega, \mc{A},\pi)^2}\right)$. Note that this function is measurable with respect to the completed $\sigma$-algebra, since $\Phi$ is separable, and for a countable dense subset $\Psi\subset\Phi$ we have
\[
\|W(x,y)\|_\mc{Z}=\sup_{\psi\in\Psi\backslash\{0\}}\frac{|\langle f,W(x,y)\rangle|}{\|\psi\|_\Phi}.
\]
Let the space of $\mc{Z}$-graphons on $(\Omega,\mc{A},\pi)^2$ be denoted by
$\mc{W}_{(\Omega,\mc{A},\pi)}$. We set
\[
\|W\|_p := \bigl\|\|W(.,.)\|_\mc{Z}\bigr\|_p.
\]
(i.e., we take the $\mc{Z}$-norm of $W(x,y)$ for every $x,y\in[0,1]$,
and then take the $L^p$-norm of the resulting function).
\end{definition}

The following notions are related to how "nice" a graphon and the underlying measure space are.

\begin{definition}
A graphon $W\in\mc{W}_{(\Omega,\mc{A},\pi)}$ is called \emph{strong} if it is also measurable with respect to the $\sigma$-algebra $\mc{A}\times\mc{A}$.\\
A graphon $W\in\mc{W}_{(\Omega,\mc{A},\pi)}$ is called \emph{Lebesguian} if the measure space $(\mc{A},\pi)$ is a standard (or Lebesgue) measure space (for a definition of standard probability spaces see e.g. \cite[Section 2.2]{Rohlin}).\\
The graphon $W\in\mc{W}_{(\Omega,\mc{A},\pi)}$ is \emph{complete} if $(\mc{A},\pi)$ is a complete measure space, and the \emph{completion} of a graphon $W\in\mc{W}_{(\Omega,\mc{A},\pi)}$ is the complete graphon $\overline{W}\in\mc{W}_{(\Omega,\overline{\mc{A}},\overline{\pi})}$ obtained through completing the measure space $(\mc{A},\pi)$.\\
Two points $x_1,x_2\in\Omega$ are called \emph{twins} with respect to the graphon $W\in\mc{W}_{(\Omega,\mc{A},\pi)}$ if $W(x_1,y)=W(x_2,y)$ for almost all $y$. The graphon $W$ is called \emph{almost twin-free} if there exists a null-set $N\subset\Omega$ such that no two points in $\Omega\backslash N$ are twins.
\end{definition}

For a $\ph\in\Phi$ and a $W\in\mc{W}_{(\Omega, \mc{A},\pi)}$ let the function $\tensor*[^\ph]W{}: \Omega^2\to\mb{R}$ be defined by
\[
\tensor*[^\ph]W{}(x,y):=\ph(W(x,y)).
\]
Note that we always have $\tensor*[^\ph]W{}\in\mc{L}_{(\Omega, \mc{A},\pi)}$.

The following lemma lets us prove weak-* measurability using only a countable dense subset of $\Phi$.

\begin{lemma}\label{le:weak*}
Let $\Psi\subset\Phi$ be a countable dense subset, and let $W:(\Omega, \mc{A},\pi)\to\mc{Z}$ be a function such that $\tensor*[^{\psi}]W{}$ is measurable for each $\psi\in\Psi$. Then $W$ is weak-* measurable.
\end{lemma}
\begin{proof}
Since $\Psi$ is countable and dense in $\Phi$, for any $\ph\in\Phi$ there exists a sequence $(\psi_n)\subset\Psi$ that converges to $\ph$ in norm. But then $\tensor*[^{\ph}]W{}$ is the pointwise limit of the measurable functions $\tensor*[^{\psi_{n}}]W{}$, and hence itself measurable.
\end{proof}

\subsection{Decorated graphs and graph densities}

We now turn our attention to decorated graphs, and the "moments" they induce for Banach space valued graphons.

If $\mc{X}$ is any set, an {\it $\mc{X}$-decorated graph} is a graph where
every edge $ij$ is decorated by an element $X_{ij}\in\mc{X}$. An
$\mc{X}$-decorated graph will be denoted by $(G,g)$, where $G$ is a
simple graph, and $g:~E(G)\to\mc{X}$.

\begin{definition}\label{def:density}
For a $\Phi$-decorated graph $\mf{F}=(F,f)$ on $k$ vertices and a $\mc{Z}$-graphon $W$, let
\[
t(\mf{F},W):=\int\limits_{x_1,\ldots,x_k\in(\Omega, \mc{A},\pi)} \prod_{ij\in E(F)} \tensor*[^{f_{ij}}]W{}(x_i,x_j) dx_1\ldots x_k.
\]
\end{definition}
Note that since $\|W\|_\mc{Z}$ lies in all $L^p$ spaces for $1\leq p<\infty$, this integral is always finite.

Our aim is to investigate to what degree graph densities determine a $\mc{Z}$-graphon. 
Just as in the real valued case, there is an inherent indeterminacy related to the choice of the underlying measure space. We therefore recall some further definitions from measure theory, and from the extension of these notions to graphons (cf. \cite[Section 2]{BCL}).
Given a graphon $W\in\mc{W}_{(\Omega, \mc{A},\pi)}$, we may obtain a graphon with the exact same graph densities by deleting a null-set from $\Omega$. To this type of indeterminacy corresponds an equivalence relation between measure spaces. Let $(\Omega,\mc{A},\pi)$ and $(\Omega',\mc{A}',\pi')$ be two probability spaces. They are called \emph{isomorphic mod $0$} if there exist null-sets $N\subset \Omega$ and $N'\subset\Omega'$ and a bijection $\mu:\Omega\backslash N\to\Omega'\backslash N'$ such that both $\mu$ and $\mu^{-1}$ are measure preserving. The map $\mu$ itself is called an \emph{isomorphism mod $0$}.\\
Another way of obtaining a new graphon with the same graph densities is by applying a "pull-back" using a measure preserving map. Let $\eta:(\Omega'',\mc{A}'',\pi'')\to(\Omega,\mc{A},\pi)$ be a measure preserving map. The \emph{pull-back} $W'':=(W)^\eta\in\mc{W}_{(\Omega'', \mc{A}'',\pi'')}$ of the graphon $W\in\mc{W}_{(\Omega, \mc{A},\pi)}$ under $\eta$ is defined through $(W)^\eta(x,y):=W(\eta(x),\eta(y))$. Then $W''$ again clearly has the same graph densities as $W$.\\
Let $U\in\mc{W}_{(\Omega_1, \mc{A}_1,\pi_1)}$ and $V\in\mc{W}_{(\Omega_2, \mc{A}_2,\pi_2)}$ be two graphons, and suppose that we have a map $\eta:\Omega_1\to\Omega_2$ that is measure preserving from the completion $\overline{\mc{A}_1}$ into $\mc{A}_2$ and satisfies $V^\eta=U$ almost everywhere. Then we say that $\eta$ is a \emph{weak isomorphism} from $U$ to $V$.\\
These isomorphism notions can be extended to the graphons themselves. We say that the graphons $W_1\in\mc{W}_{(\Omega_1, \mc{A}_1,\pi_1)}$ and $W_2\in\mc{W}_{(\Omega_2, \mc{A}_2,\pi_2)}$ are \emph{isomorphic mod $0$} if there exists an isomorphism mod $0$ $\mu:\Omega_1\to\Omega_2$ such that $(W_2)^\mu=W_1$ almost everywhere. As a short-hand notation we write $W_1\cong W_2$.\\
We say that $W_1$ and $W_2$ are \emph{weakly isomorphic} if there exists a graphon $W_3\in\mc{W}_{(\Omega_3, \mc{A}_3,\pi_3)}$ and weak isomorphisms from each of $W_1$ and $W_2$ into $W_3$. Note that it is not immediately clear that being weakly isomorphic is an equivalence relation.

\subsection{Main result}

Using the notations and definitions introduced above, the main result of this paper can be formulated as follows.\\

\begin{thm}\label{thm:main}\leavevmode\\
\begin{enumerate}
\item Let $W\in\mc{W}_{(\Omega, \mc{A},\pi)}$ and $W'\in\mc{W}_{(\Omega', \mc{A}',\pi')}$ be almost twin-free strong Lebesguian graphons. Further assume that the $p$-norms of $W$ satisfy:
\begin{eqnarray*}
\sum_{n=1}^\infty \|W\|_{2nk}^{-k}=\infty
\end{eqnarray*}
for all $k\in\mb{N}^+$.
Then
\[
t(\mf{F},W)=t(\mf{F},W')
\]
for every $\Psi$-decorated simple graph $\mf{F}$ if and only if $W$ and $W'$ are isomorphic mod $0$.
\item  Let $W\in\mc{W}_{(\Omega, \mc{A},\pi)}$ and $W'\in\mc{W}_{(\Omega', \mc{A}',\pi')}$ be general graphons. Further assume that the $p$-norms of $W$ satisfy:
\begin{eqnarray*}
\sum_{n=1}^\infty \|W\|_{2nk}^{-k}=\infty
\end{eqnarray*}
for all $k\in\mb{N}^+$. 
Then
\[
t(\mf{F},W)=t(\mf{F},W')
\]
for every $\Psi$-decorated simple graph $\mf{F}$ if and only if $W$ and $W'$ are weakly isomorphic.
\end{enumerate}
\end{thm}

Here we note that the family of Carleman-type conditions required can be viewed as requiring that not only the function/random variable $\|W(\cdot,\cdot)\|$, but all of its powers should satisfy the Carleman condition in order to obtain a moment determinacy result. 
This may seem superfluous, but the bounds are actually used for determinacy of new random variables induced by partially decorated graphs (see Section \ref{sect:coupling}). These moments do not correspond to expectations of simple powers of the same random variable. In addition, random variables whose moments are bounded from above by (or even equal to) the moments of a power of a moment determinate random variable need not be moment determinate themselves, even for the most common distributions we know (see, e.g., the papers \cite{Berg} by Berg, and \cite{Lin} by Lin and Huang, or the monograph \cite{Stoyanov} by Stoyanov).

The above theorem is a generalization of \cite[Theorem 2.1]{BCL}, since in the case of bounded real valued functions $W$, the Carleman conditions are automatically satisfied.

\section{Graphon constructions}

To prove the second part of our main theorem, we shall - following Borgs, Chayes and Lov\'asz \cite{BCL} - be transforming our original graphon by changing the underlying measure space in several steps until we end up with a standard Lebesguian space, thereby reducing the problem to part $(i)$. This will be achieved mainly through manipulating the corresponding $\sigma$-algebra and adapting the graphon to these successive changes.
We therefore briefly recall the necessary notions from measure theory.

A set $\mc{S}$ of subsets of $\Omega$ induces a partition $\mc{P}[\mc{S}]$ of $\Omega$ through the natural equivalence relation $\omega_1\sim\omega_2 \Leftrightarrow \left[\forall S\in\mc{S}: (\omega_1\in S \wedge \omega_2\in S)\vee(\omega_1\not\in S \wedge \omega_2\not\in S)\right]$. This is the finest partition for which $\mc{S}$ separates the classes. A graphon $W\in\mc{W}_{(\Omega,\mc{A},\pi)}$ is said to be \emph{separating} if $\mc{A}$ separates $\Omega$, or in other words if $W=W/\mc{P}[\mc{A}]$.

A $\sigma$-algebra $\mc {A}$ is said to be \emph{countably generated} if there is a countable set $S\subset\mc{A}$ such that the generated $\sigma$-algebra satisfies $\sigma(S)=\mc{A}$. A set $S\subset\mc{A}$ is said to be a \emph{basis} of the measure space $(\Omega,\mc{A},\pi)$ if $\sigma(S)$ is \emph{dense} in $\mc{A}$, that is, if for every $A_1\in\mc{A}$ there exists an $A_2\in\sigma(S)$ such that $\pi(A_1\triangle A_2)=0$. A graphon $W\in\mc{W}_{(\Omega,\mc{A},\pi)}$ is said to be countably generated if $\mc{A}$ itself is.

A probability space $(\Omega',\mc{A}',\pi')$ is said to be a \emph{full subspace} of the probability space $(\Omega,\mc{A},\pi)$ if $\Omega'\subset\Omega$ has outer measure $1$ (it need not be measurable) and $(\mc{A}',\pi')$ is the restriction of $(\mc{A},\pi)$ to $\Omega'$, i.e., $\mc{A}'=\left\{A\cap\Omega'|A\in\mc{A}\right\}$ and $\pi'(A\cap\Omega')=\pi(A)$ for all $A\in\mc{A}$.\\
A measure preserving map $\mu:(\Omega,\mc{A},\pi)\to(\Omega',\mc{A}',\pi')$ between two probability spaces is called an \emph{embedding} of $(\Omega,\mc{A},\pi)$ into $(\Omega',\mc{A}'\pi')$ if it is an isomorphism between the former and a full subspace of the latter. Given two graphons $W\in\mc{W}_{(\Omega,\mc{A},\pi)}$ and $W'\in~\mc{W}_{(\Omega',\mc{A}',\pi')}$, the embedding $\mu$ is said to be an \emph{embedding of $W$ into $W'$} if in addition $(W')^\mu=W$ almost everywhere.

Let us start with the following Lemma, that allows us to change the $\sigma$-algebra we work with to a countably generated one one, without losing measurability.

\begin{lemma}\label{le:countable}
Let $(\Omega,\mc{A})$ and $(\Omega',\mc{A}')$ be measurable spaces, and let $W:\Omega\times\Omega'\to\mc{Z}$ be a weak-* measurable function with respect to the $\sigma$-algebra $(\mc{A}\times\mc{A}')$. Then there exist countably generated $\sigma$-algebras $\mc{A}_0\subset\mc{A}$ and $\mc{A}_0'\subset\mc{A}'$ such that $W$ is weak-* measurable with respect to $(\mc{A}_0\times\mc{A}'_0)$.
\end{lemma}
\begin{proof}
Choose a dense countable subset $\Psi\subset\Phi$. For each $\psi\in\Psi$ we can find countable sets $\mc{S}_\psi\subset\mc{A}$ and $\mc{S}'_\psi\subset\mc{A}'$ such that $\psi\circ W$ is measurable w.r.t. the generated $\sigma$-algebra (cf. \cite[Lemma 3.4]{BCL}, boundedness is actually not needed, since we can compose with the $\arctan$ function to reduce to the bounded case). Thus taking $\mc{A}_0$ and $\mc{A}_0'$ to be the sub-$\sigma$-algebras generated by $(\mc{S}_\psi)_{\psi\in\Psi}$ and $(\mc{S}'_\psi)_{\psi\in\Psi}$, respectively, Lemma \ref{le:weak*} ensures the required measurability of $W$.
\end{proof}

Next we shall introduce two further constructions that given a graphon allow us to create a new graphon, preserving some of its essential properties.

\begin{remark}
For ease of notation, using Proposition \ref{prop:RN}, the integrals we write from here on are to be understood in the weak-* sense rather than the strong/Bochner sense.
\end{remark}

\begin{lemma}\label{le:push-forward}
Let $(\Omega,\mc{A},\pi)$ and $(\Omega',\mc{A}',\pi')$ be probability spaces, let $\tau:\Omega\to\Omega'$ be a measure preserving map, and let $W\in\mc{W}_{(\Omega, \mc{A},\pi)}$.
\begin{enumerate}
\item There exists a 
strong graphon $W'\in\mc{W}_{(\Omega', \mc{A}',\pi')}$
that satisfies
\begin{equation}\label{eqn:push-forward}
\int_{A_1'\times A_2'} W_\tau(x',y')d\pi'(x')d\pi'(y')=\int_{\tau^{-1}(A_1')\times\tau^{-1}(A_2')}W(x,y)d\pi(x)d\pi(y)
\end{equation}
for all $A_1',A_2'\in\mc{A}'$.
\item If $\tau$ is an embedding, then $(W_\tau)^\tau=W$ almost everywhere.
\end{enumerate}
\end{lemma}

\begin{proof}
First let $\mc{A}'_\tau:=\tau^{-1}(\mc{A}')\subset\mc{A}$ and define $\widetilde{W}:=E(W|\mc{A}'_\tau\times\mc{A}'_\tau)$. Also let $\pi_\tau:=\pi|_{\mc{A}'_\tau}$. Then define a measure $\mu$ on $\mc{A}'_\tau\times\mc{A}'_\tau$ through
\[
\mu(A_1\times A_2)=\int_{A_1\times A_2} W(x,y) d\pi(x)d\pi(y)
\]
for all $A_1,A_2\in\mc{A}'_\tau$. Since $\|W\|$ lies in $L^1$, we have that $\mu\ll \pi\times\pi$. Let $\mu_\tau$ be the push-forward of $\mu$. Since $\ph$ is measure-preserving, $\mu_\ph$ is a $\mc{Z}$-valued measure on $\mc{A}'$ that is absolutely continuous with respect to $\pi'\times\pi'$.

By Lemma \ref{le:conditional}, there exists a weak-* integrable function $W_\tau:\Omega'\times\Omega'\to\mc{Z}$ such that for all $A'_1,A'_2\in\mc{A}'$,
\[
\mu(A'_1\times A'_2)=\int_{A'_1\times A'_2} W'(x,y) d\pi'(x)d\pi'(y).
\]
This $W'$ then clearly satisfies equation \ref{eqn:push-forward}, and since the push-forward $\mu'$ of the symmetric measure $\mu$ is itself symmetric, it follows that $W'$ is too.
 By the norm inequality in Lemma \ref{le:conditional}, we actually have that $W'\in\mc{W}_{(\Omega', \mc{A}',\pi')}$, completing the proof of part (i).

By construction $(W_\tau)^\tau=\widetilde{W}=E(W|\mc{A}'_\tau\times\mc{A}'_\tau)$, so part (ii) follows from the definition of $\mc{A}'_\tau$.
\end{proof}

\begin{definition}
The function $W_\tau$ defined in \ref{le:push-forward} is called the \emph{push-forward} of $W$ under $\tau$.
\end{definition}

Given a graphon $W\in\mc{W}_{(\Omega,\mc{A},\pi)}$ and a partition $\mc{P}$ of $\Omega$, we can use the push-forward construction to define the \emph{factor graphon} $W/\mc{P}$ of $W$ under $\mc{P}$. Consider the surjection $\mu:\omega\mapsto[\omega]$ from $\Omega$ to $\Omega/\mc{P}$, where $[\omega]$ denotes the partition class of $\omega$. Define the measure space $(\Omega/\mc{P},\mc{A}/\mc{P},\pi/\mc{P})$ as the push-forward of $(\Omega,\mc{A},\pi)$ under $\mu$. Then $\mu$ is automatically measure preserving, and we let $W/\mc{P}:=W_\mu\in \mc{W}_{(\Omega/\mc{P},\mc{A}/\mc{P},\pi/\mc{P})}$.

Note that push-forwards being a type of conditional expectations, the moments of a graphon are usually not preserved. However, if the underlying measure space is carefully manipulated, this problem can be avoided, as illustrated by the next theorem, which sums up the different steps that will allow us to pass from part (i) to part (ii) of Theorem \ref{thm:main}.

\begin{thm}
Let $W\in\mc{W}_{(\Omega, \mc{A},\pi)}$ be a graphon.
\begin{enumerate}
\item One can change the value of $W$ on a set of $\pi\times\pi$-measure 0 to get a strong $\mc{Z}_1$-valued graphon.
\item Suppose that $W$ is a strong graphon. Then there exists a countably generated $\sigma$-algebra $\mc{A}'\subset\mc{A}$ such that $W$ is weak-* measurable with respect to $(\mc{A}'\times\mc{A}')$.
\item The graphon $W/\mc{P}[\mc{A}]$ is separating. If $W$ is countably generated, then so is $W/\mc{P}[\mc{A}]$.
\item  Suppose that $W$ is a separating graphon on a probability space with a countable basis. Then the completion of $W$ can be embedded in a Lebesguian graphon.
\item  Suppose that $W$ is a strong graphon, and let $\mc{P}$ be the partition into the twin-classes of $W$. Then $W/\mc{P}$ is almost twin-free. If $W$ is Lebesguian, then $W/\mc{P}$ is Lebesguian as well. Furthermore the projection $W\to W/\mc{P}$ is a weak isomorphism.
\end{enumerate}
\end{thm}

For part (i), consider the conditional expectation $W':=E(W\left|\mc{A}\times\mc{A}\right.)$. Since $W$ and the underlying $\sigma$-algebra both are symmetric, we may assume that so is $W'$. By construction, $W'$ is measurable with respect to $\mc{A}\times\mc{A}$, so it is enough to show that $W=W'$ almost everywhere. To this end note that $\int_{A_1\times A_2}(W'-W)=0$ for all $A_1,A_2\in\mc{A}$, and therefore also for all $S\in\overline{\mc{A}\times\mc{A}}$ we have $\int_S (W'-W)=0$, implying $W'-W=0$ almost everywhere.

Part (ii) is an easy consequence of Lemma \ref{le:countable}.

For part (iii), note that by identifying elements in the came class of the partition $\mc{P}[\mc{A}]$, we obtain a $\sigma$-algebra that is isomorphic to $\mc{A}$ under the natural map.

For part (iv), let $\mc{S}$ be a countable set generating $\mc{A}$. Since any separating complete probability space with a countable basis can be embedded into a Lebesgue space (see e.g. \cite[Section 2.2]{Rohlin}) and $\mc{S}$ is a countable basis for $(\overline{\mc{A}},\overline{\pi})$, there exists an embedding $\eta:(\Omega,\overline{\mc{A}},\overline{\pi})\to(\Omega',\overline{\mc{A}'},\overline{\pi'})$
where $(\Omega',\overline{\mc{A}'},\overline{\pi'})$ is a Lebesgue space. Consider the push-forward $W':=W_\eta$. By Lemma \ref{le:push-forward}(ii), the mapping $\eta$ is an embedding of the completion $\overline{W}$ of $W$ into the Lebesguian graphon $W'$.

For part (v), we may by part (ii) assume that $\mc{A}$ is countably generated, since the relation of being twins is the same for any underlying $\sigma$-algebra $\mc{B}\times\mc{B}$ that leaves $W$ weak*-measurable (and hence weak-* integrable).
Let $\mc{A}_\mc{P}$ denote the sub-$\sigma$-algebra of $\mc{A}$ consisting of the sets that do not separate any pair of twins. Note that by the construction in Lemma \ref{le:push-forward}, we then have that $(W/\mc{P})^\tau=E(W|\mc{A}_\mc{P}\times\mc{A}_\mc{P})\in\mc{W}_{(\Omega, \mc{A}_\mc{P},\pi)}$. Thus the projection $W\to W/\mc{P}$ is a weak isomorphism if $W=E(W|\mc{A}_\mc{P}\times\mc{A}_\mc{P})$ almost everywhere.

Let $\widehat{W}:=E(W|\mc{A}_\mc{P}\times\mc{A}_\mc{P})$. Since $\Psi\subset\Phi$ is countable and dense, hence separates elements in $\mc{Z}$, it is sufficient to show that for any $A,B\in\mc{A}$ and $\psi\in\Psi$ we have
\[
\int_{A\times B}\tensor*[^\psi]W{}(x,y) d\pi(x) d\pi(y)=\int_{A\times B} \tensor*[^\psi]{\widehat{W}}{}(x,y)d\pi(x)d\pi(y).
\]

This, and the fact that $W/\mc{P}$ is twin-free, can easily be proven, and we refer to \cite[Section 3.3.5]{BCL} for the details.

We now wish to show that $W/\mc{P}$ is Lebesguian under the assumption that $W$ itself is. By \cite[Section 3.2]{Rohlin} the measure space $(\Omega/\mc{P},\mc{A}/\mc{P},\pi/\mc{P})$ is a Lebesgue space if there exists a countable set $\mathcal{S}\subset\mc{A}$ that separates points if and only if they are from different partition classes. Let $\mc{T}$ be a countable set generating $\mc{A}$, closed under finite intersections. For each $A\in\mc{A}$ and $x\in\Omega$, let
\[
\mu_x(A)=\int_A W(x,y) d\pi(y).
\]
Since $W$ is weak-* integrable with respect to $\mc{A}\times\mc{A}$, the mapping $A\mapsto\mu_x(A)$ is a $\mc{Z}$-valued measure for all $x\in\Omega$, and $x\mapsto\mu_x(A)$ is an $\mc{A}$-measurable function on $\Omega$ for each $A\in\mc{A}$.
Note that then $x,x'$ are twins if and only if $\mu_x(A)=\mu_{x'}(A)$ for all $A\in\mc{A}$, and since each measure $\mu_x(\cdot)$ is uniquely determined by the values taken on $\mc{T}$, $x$ and $x'$ are twins if and only if $\mu_x(T)=\mu_{x'}(T)$ for all $T\in\mc{T}$.

 Let $\Psi\subset\Phi$ be a countable dense set, and for every $\psi\in\Psi$, $T\in\mc{T}$ and rational number $r$ let $S_{\psi,T,r}:=\{x\in\Omega:\psi(\mu_x(T))\geq r\}$. These are countably many and clearly do not separate twins. If however $x$ and $x'$ are not twins, then for some $T\in\mc{T}$ we have $\mu_x(T)\neq\mu_{x'}(T)$. Since $\Psi$ separates $\mc{Z}$, we can then find $\psi\in\Psi$ and $r\in\mb{Q}$ such that $S_{\psi,T,r}$ separates $x$ and $x'$.

\begin{cor}\label{cor:Lebesgue}
Every graphon $W\in\mc{W}_{(\Omega, \mc{A},\pi)}$ has a weak isomorphism into an almost twin-free strong Lebesguian $\mc{Z}$-valued graphon $\widetilde{W}$. In addition $\|W\|_p=\left\|\widetilde{W}\right\|_p$ for all $1\leq p<\infty$.
\end{cor}
\begin{proof}
First let us change the graphon $W\in\mc{W}_{(\Omega, \mc{A},\pi)}$ on a null-set to obtain 
a strong graphon $W_1$ applying (i) of the previous theorem. Then the identity map on the underlying probability space $(\Omega, \mc{A},\pi)$ will be a $\overline{\mc{A}}-\overline{\mc{A}}$-measurable weak isomorphism $\tau_1$ from $W$ to $W_1$. 
Now by points (ii)-(iv) the completion of the graphon $W_1$ has an embedding $\tau_2$ into a Lebesguian graphon $W_3$, which in turn by point (v) can be projected onto its almost twin-free Lebesguian form $\widetilde{W}\in\mc{W}_{(\widetilde{\Omega}, \widetilde{\mc{A}},\widetilde{\pi})}$ through a weak isomorphism $\tau_3$. Since the composition $\tau:=\tau_3\circ\tau_2\circ\tau_1$ is a $\overline{\mc{A}}-\widetilde{\mc{A}}$-measurable mapping between $(\Omega, \mc{A},\pi)$ and the Lebesgue space $(\widetilde{\Omega}, \widetilde{\mc{A}},\widetilde{\pi})$, it is indeed a weak isomorphism from $W$ into $\widetilde{W}$. Concerning the last assertion, note that none of the transformations in the previous theorem actually changed the norm of the graphon.

\end{proof}

\section{Anchored graphons}
As mentioned in the introduction, two-variables functions on a product measure space $\Omega\times\Omega$ lack a canonical form, or a canonical reordering of the underlying product  space. To counter this, Borgs, Chayes and Lov\'asz made in \cite{BCL} use of countable sequences of random elements of $\Omega$ and defined what they called \emph{canonical ensembles}. In essence this means that instead of having a single well-defined canonical form, one rather has a whole family of functions, together with a probability measure on said family. This pair contains all the relevant information on the original function, and is easier to handle than the original single function. In addition these new functions all live on the same $\sigma$-algebra, independently of the measure space of the original graphon, making it possible to compare these random representations.

First we shall need a few results that guarantee that with probability 1, the randomness we wish to introduce does not interfere with measurability.

\begin{lemma}\label{le:random-meas}
Let $(\Omega,\mc{A},\pi)$ and $(\Omega',\mc{A}',\pi')$ be probability spaces, and let $W:\Omega\times\Omega'\to\mc{Z}$ be a weak-* measurable function with respect to $\mc{A}\times\mc{A}'$ such that $\|W\|\in L^2(\Omega\times\Omega')$. Let further $Y_1,Y_2,\ldots$ be independent random points from $\Omega'$, and $\mc{A}_0\subseteq\mc{A}$ the random $\sigma$-algebra generated by the functions $W(\cdot,Y_k)$. Then with probability 1, $W$ is almost weak-* measurable with respect to $\mc{A}_0\times\mc{A}'$.
\end{lemma}

\begin{proof}
By Lemma \ref{le:countable}, we may assume that both $\mc{A}$ and $\mc{A}'$ are countably generated. Let $\mc{A}_1\subset\mc{A}_2\subset\ldots$ and $\mc{A}'_1\subset\mc{A}'_2\subset\ldots$ be sequences of finite $\sigma$-algebras such that $\sigma(\cup_n\mc{A}_n)=\mc{A}$ and $\sigma(\cup_n\mc{A}'_n)=\mc{A}'$, and let $P_n'$ denote the partition of $\Omega'$ into the atoms of $\mc{A}'_n$. For each $y\in S\in P'_n$ with $\pi'(S)>0$ let
\[
U_{n,m}(x,y):=\frac{1}{m\pi'(S)}\sum_{\substack{j\leq m \\  Y_j\in S}} W(x,Y_j),
\]
whilst 
$U_{n,m}(x,y):=0$ whenever $\pi'(S)=0$. Let $\Psi\subset\Phi$ be a countable dense subset of the predual of $\mc{Z}$.

We first wish to prove that for every $\psi\in\Psi$, $n\geq1$, every $A\in\cup_n\mc{A}_n$ and $A'\in\mc{A}'_n$ we have with probability 1 that
\begin{equation}\label{eqn:random-conv}
\int_{A\times A'}\tensor*[^\psi]{U}{_{n,m}} \,d\pi d\pi' \xrightarrow{m\to\infty} \int_{A\times A'} \tensor*[^\psi]W{} d\pi d\pi'.
\end{equation}
Since $\mc{A}'_n$ is atomic, it is enough to show this for sets $A'=S\in P'_n$ with $\pi'(S)>0$. For every $y_0\in S$ we then have
\[
\int_A \tensor*[^\psi]{U}{_{n,m}}(x,y_0)\, d\pi(x)=\frac{1}{m\pi'(S)}\sum_{\substack{j\leq m \\  Y_j\in S}} \int_A \tensor*[^\psi]W{}(x,Y_j) d\pi(x).
\]
Since $\|W\|$ lies in $L^2\subset L^1$, the function $\tensor*[^\psi]W{}$ lies in $L^1$, hence we may apply the Law of Large Numbers to obtain
\[
\int_A \tensor*[^\psi]{U}{_{n,m}}(x,y_0) \,d\pi(x)\xrightarrow{m\to\infty}\frac{1}{\pi'(S)}\int_{A\times S} \tensor*[^\psi]W{} d\pi d\pi'.
\]
Since both sides are independent of the choice of $y_0\in S$, we may integrate it out over $S$ to obtain equation (\ref{eqn:random-conv}).

We have a countable number of choices for $\psi\in\Psi$, $n\in\mb{N}^+$, $A'\in\mc{A}'$ and $A\in\cup_k\mc{A}_k$, hence with probability 1 equation (\ref{eqn:random-conv}) holds for all $\psi\in\Psi$, all $n\in\mb{N}^+$, $A'\in\mc{A}_n'$ and $A\in\cup_k\mc{A}_k$.
Also, note that
\[
\|U_{n,m}(x,y)\|\leq\frac{1}{m\pi'(S)}\sum_{\substack{j\leq m \\  Y_j\in S}} \|W(x,Y_j)\|,
\]
whence by AM-QM,
\begin{equation*}
\|U_{n,m}(x,y)\|^2\leq\frac{1}{m\pi'(S)^2}\sum_{\substack{j\leq m \\  Y_j\in S}} \|W(x,Y_j)\|^2.
\end{equation*}
For each $y_0\in S\in P_n'$ with $\pi'(S)>0$ we then have
\begin{equation}\label{eqn:ineq}
\int_\Omega \|U_{n,m}(x,y_0)\|^2 \,d\pi(x)\leq \frac{1}{\pi'(S)^2}\frac1m \sum_{\substack{j\leq m \\  Y_j\in S}} \int_\Omega \|W(x,Y_j)\|^2 d\pi(x).
\end{equation}
Again inequality (\ref{eqn:ineq}) is independent of the choice of $y_0\in S$, so we may integrate over it to obtain
\[
\int_{\Omega\times S} \|U_{n,m}(x,y_0)\|^2 \,d\pi(x)d\pi'(y_0)\leq \frac{1}{\pi'(S)}\frac1m \sum_{\substack{j\leq m \\  Y_j\in S}} \int_\Omega \|W(x,Y_j)\|^2 d\pi(x).
\]

Since by assumption $\|W\|$ lies in $L^2$, we may again apply the Law of Large Numbers to the right hand side to obtain that with probability 1
\[
\frac{1}{\pi'(S)}\frac1m \sum_{\substack{j\leq m \\  Y_j\in S}} \int_\Omega \|W(x,Y_j)\|^2 d\pi(x)\xrightarrow{m\to\infty} \frac{1}{\pi'(S)} \int_{\Omega\times S}\|W\|^2 d\pi d\pi'<\infty.
\]

Since there are finitely many $S\in P'_n$ with $\pi'(S)>0$, summing over all such $S$ yields that with probability 1, the sequence $(U_{n,m})_{m\in\mb{N}^+}$ is uniformly bounded in $L^2$-norm for each $n\in\mb{N}^+$.

From now on assume that the choice of the $Y_j$ is such that equation (\ref{eqn:random-conv}) holds for all $\psi\in\Psi$, all $n\in\mb{N}^+$, $A'\in\mc{A}'_n$ and $A\in\cup_n\mc{A}_n$, and that the sequence $(U_{n,m})_{m\in\mb{N}^+}$ is uniformly bounded in $L^2$-norm for each $n\in\mb{N}^+$.
For a fixed $n$, by compactness and since there are countably many elements in $\Psi$, we may by a diagonal argument choose a subsequence $(m_k)_{k\in\mb{N}^+}\subset\mb{N}^+$ such that for each $\psi\in\Psi$ the sequence $(\tensor*[^\psi]U{_{n,m_k}})_{k\in\mb{N}^+}$ converges weakly in $L^2(\mc{A}\times\mc{A}_n',\mb{R})$. Denote the weak limits by $U_{n,\psi}$. Since by construction each $\tensor*[^\psi]U{_{n,m_k}}$ is $\mc{A}_0\times\mc{A}_n'$-measurable, the same is true for the weak limits $U_{n,\psi}$.

By equation (\ref{eqn:random-conv}) we have that for each $n\geq1$ and $\psi\in\Psi$,
\begin{eqnarray*}
\int_{A\times A'} U_{n,\psi} \,d\pi d\pi'&=&\lim_{k\to\infty} \int_{A\times A'} \tensor*[^\psi]U{_{n,m_k}} d\pi d\pi'=\int_{A\times A'} \tensor*[^\psi]W{} d\pi d\pi'\\
&=&\int_{A\times A'} E(\tensor*[^\psi]W{}|\mc{A}\times\mc{A}_n') \,d\pi d\pi'
\end{eqnarray*}
for all $A\in\cup_n\mc{A}_n$ and $A'\in\mc{A}'_n$.
But both functions $U_{n,\psi}$ and $ E(\tensor*[^\psi]W{}|\mc{A}\times\mc{A}_n')$ lie in $L^2(\mc{A}\times\mc{A}_n',\mb{R})$, and since $\cup_n\mc{A}_n$ is dense in $\mc{A}$, the above equality is in fact true for all $A\in\mc{A}$.
This means that the functions $U_{n,\psi}$ and $ E(\tensor*[^\psi]W{}|\mc{A}\times\mc{A}_n')$ represent the same element in $L^2(\mc{A}\times\mc{A}_n',\mb{R})$. Since $U_{n,\psi}$ is $\mc{A}_0\times\mc{A}_n'$-measurable, this in turn means that for every $\psi\in\Psi$
\begin{equation}\label{eqn:upward}
 E(\tensor*[^\psi]W{}|\mc{A}_0\times\mc{A}_n')= E(\tensor*[^\psi]W{}|\mc{A}\times\mc{A}_n')
\end{equation}
in $L^2(\mc{A}\times\mc{A}_n',\mb{R})$. By Levy's Upward Theorem, however, the left hand side of (\ref{eqn:upward}) tends a.e. to $E(\tensor*[^\psi]W{}|\mc{A}_0\times\mc{A}')$ as $n\to\infty$, whilst the right hand side tends a.e. to $E(\tensor*[^\psi]W{}|\mc{A}\times\mc{A}')$.
Thus for each $\psi\in\Psi$ we have that
\[
 E(\tensor*[^\psi]W{}|\mc{A}_0\times\mc{A}')= E(\tensor*[^\psi]W{}|\mc{A}\times\mc{A}')=\tensor*[^\psi]W{}
\]
in $L^2(\mc{A}\times\mc{A}',\mb{R})$, meaning that $\tensor*[^\psi]W{}$ is almost $\mc{A}_0\times\mc{A}'$-measurable for all $\psi\in\Psi$. By Lemma \ref{le:weak*}, and since $\Psi$ is countable, it then follows that $W$ is almost weak-* measurable with respect to $\mc{A}\times \mc{A}'$.

\end{proof}

The following is an immediate consequence of the previous theorem.
\begin{cor}
Let $(\Omega,\mc{A},\pi)$ and $(\Omega',\mc{A}',\pi')$ be probability spaces, and let $W:\Omega\times\Omega'\to\mc{Z}$ be a weak-* measurable function with respect to $\mc{A}\times\mc{A}'$ with $\|W\|\in L^2$. Let further $\mc{A}_0\subset\mc{A}$ be a sub-$\sigma$-algebra. If $W(\cdot,y)$ is weak-* measurable with respect to $\mc{A}_0$ for almost all $x\in\Omega$, then $W$ is almost weak-* measurable with respect to $\mc{A}_0\times\mc{A}'$.
\end{cor}

Applying the previous corollary twice, we obtain the following result, which is the key element in our randomization.

\begin{cor}
Let $W\in\mc{W}_{(\Omega,\mc{A},\pi)}$ be a strong graphon, and let $X_1,X_2,\ldots$ be independent random points from $\Omega$. Let $\mc{A}_0\subset\mc{A}$ be the random $\sigma$-algebra generated by the functions $W(\cdot,X_k)$, $k\in\mb{N}^+$. Then with probability 1, $W$ is almost weak-* measurable with respect to $\mc{A}_0\times\mc{A}_0$.
\end{cor}
\begin{proof}
The proof of \cite[Corollary 4.3]{BCL} works also in this more general setting, with minor natural modifications.
\end{proof}

Let now $W\in\mc{W}_{(\Omega,\mc{A},\pi)}$ be a strong graphon and $\alpha=(\alpha_1,\alpha_2,\ldots)$ an infinite sequence of points in $\Omega$. Consider the map $x\mapsto \Gamma_\alpha(x)$ defined by
\begin{equation}\label{eqn:tag}
\Gamma_\alpha(x):=((\tensor*[^\psi]W{}(x,\alpha_1))_{\psi\in\Psi},(\tensor*[^\psi]W{}(x,\alpha_2))_{\psi\in\Psi},\ldots)\in \left(\mb{R}^\Psi\right)^{\mb{N}^+}=:\mc{R}.
\end{equation}
Since each $\tensor*[^\psi]W{}(\cdot,\alpha_j)$ is measurable by assumption, $\Gamma_\alpha$ is a measurable map from $\Omega$ into $\mc{R}$ with respect to the standard Borel $\sigma$-algebra $\mc{K}$ on the product space, and thereby defines a push-forward measure $\kappa_\alpha$ on $\mc{K}$ through
\begin{equation}\label{eqn:measure}
\kappa_\alpha(S)=\pi(\Gamma^{-1}_\alpha(S)),
\end{equation}
for $S\in\mc{K}$.
Denote by $\mc{L}$ the completion of $\mc{K}$ with respect to $\kappa_\alpha$, and $\lambda_\alpha$ the extension of $\kappa_\alpha$ to $\mc{L}$.

Let further $W_{\Gamma_\alpha}$ be the push-forward of $W$ under $\Gamma_\alpha$, and denote by $W_\alpha\in\mc{W}_{(\mc{R},\mc{L},\lambda_\alpha)}$ its completion.
\begin{definition}
The graphon $W_\alpha$ is called the \emph{anchored graphon} with respect to the \emph{anchor sequence} $\alpha$. An anchor sequence $\alpha$ is called \emph{regular} if $W=(W_{\Gamma_\alpha})^{\Gamma_\alpha}$ almost everywhere.
\end{definition}
\begin{remark}
Since it is the product of countably many copies of $\mb{R}$, $\mc{R}$ itself is a complete Polish space, hence $W_\alpha$ is always a Lebesguian graphon.
\end{remark}

\begin{lemma}\label{le:regular}
Almost all $\alpha\in\Omega^{\mb{N}^+}$ are regular.
\end{lemma}
\begin{proof}
Let $\mc{A}_\alpha$ denote the pullback of the $\sigma$-algebra $\mc{K}$ under the map $\Gamma_\alpha$. Since $\Gamma_\alpha$ is measurable, we have $\mc{A}_\alpha\subset\mc{A}$. Also, by construction $\mc{A}_\alpha$ is the smallest $\sigma$-algebra such that all of the functions $\tensor*[^\psi]W{}(\cdot,\alpha_j)$ ($\psi\in\Psi$, $1\leq j$) are measurable. We may thus apply Lemma \ref{le:random-meas} to obtain that for almost all $\alpha$, $W$ is almost weak-* measurable with respect to $\mc{A}_\alpha\times \mc{A}_\alpha$. But by the construction given in the proof of Lemma \ref{le:push-forward} it immediately follows that $W=(W_{\Gamma_\alpha})^{\Gamma_\alpha}$ a.e..
\end{proof}

\section{Densities and coupling of anchored graphons}\label{sect:coupling}

Now that we have introduced random canonical forms for graphons, our aim is to show that on the one hand two weakly isomorphic graphons have essentially the same random canonical form, and on the other hand that this random canonical form is determined by the moment functions. The main helping tool here are densities stemming from so-called partially labeled graphs. These provide the necessary randomness that will act as a bridge between the usual deterministic densities and the random anchored graphons, and allow us to find appropriate couplings between the anchors. We present the necessary notions and properties necessary in the context of decorated graphs, and refer to \cite{Hombook} for futher properties and applications of partial labeling.

 Let $\mf{F}=(F,f)$ be a decorated graph. A \emph{partial labeling} of $\mf{F}$ is an injective map from a subset of the vertices of $F$ into the set $\mb{N}^+$ of positive integers. If the injective map has as image the set $\{1,2,\ldots,k\}$, the partially labeled decorated graph is called \emph{$k$-labeled}. To simplify notation, the case $k=0$ corresponds to unlabeled decorated graphs.

Two partially labeled decorated graphs $\mf{F}_1$ and $\mf{F}_2$ are \emph{isomorphic} if there exists a graph isomorphism between $F_1$ and $F_2$ that preserves both the labels and the decorations. The \emph{product} $\mf{F}_1\mf{F}_2$ of two partially labeled decorated graphs $\mf{F}_1$ and $\mf{F}_2$ is itself a partially labeled decorated graph, defined as follows: take the disjoint union of $F_1$ and $F_2$, then merge the vertices that have identical labels, whilst keeping the labels and decorations as well as any multiple edges that may arise.

Next, we define marginals induced by partial labelings.
Supose $\mf{F}=(F,f)$ is a partially labeled $\Phi$-decorated graph with vertex set $V(F)=\left\{v_1,\ldots,v_k\right\}$, where the vertices $v_1,\ldots,v_r$ are labeled by the positive integers $\ell_1,\ldots,\ell_r$, and the remaining vertices are unlabeled.
Given a graphon $W\in\mc{W}_{(\Omega,\mc{A},\pi)}$ and an infinite sequence $\beta=(\beta_1,\beta_2,\ldots)$ in $\Omega$, for $1\leq j\leq r$ set $X_j:=\beta_{\ell_j}$, and let $X_{r+1},\ldots,X_k$ be independent random points of $\Omega$ from the distribution $\pi$. Then we may define the marginal
\[
t_\beta(\mf{F},W):=E\left(\prod_{v_iv_j\in E(F)}\tensor*[^{f_{v_iv_j}}]W{}(X_i,X_j)\right).
\]

This marginal only depends on the finitely many elements of $\beta$ whose index appears as label, so it will be convenient to sometimes omit the tail of $\beta$ containing no labels.\\ If $\mf{F}_1$ and $\mf{F}_2$ are two $k$-labeled graphs, it can easily be seen that
\[
t(\mf{F}_1\mf{F}_2)=\int_{\Omega^k} t_{x_1\ldots x_k}(\mf{F}_1,W)t_{x_1\ldots x_k}(\mf{F}_2,W) d\pi(x_1)\ldots d\pi(x_k).
\]

We shall show how anchor sequences can be used to prove almost everywhere equality of appropriate graphons. Having to involve densities of labeled \emph{multigraphs} is a technical necessity of the approach introduced in \cite{BCL}, as we wish to prove equality of measures through equality of moments, and higher mixed moments in this context naturally correspond to densities of multigraphs rather than those of simple graphs. As we shall later see, this is not going to be a hindrance.

Given a countably generated dense subspace $\Psi\in\Phi$ of decorations, let $\mc{F}^*_k$ denote the set of $k$-labeled $\Psi$-decorated multigraphs with no edge between labeled vertices, and let $\mc{F}^*:=\cup_{n=0}\mc{F}^*_k$.

\begin{lemma}\label{le:equality}
Let $W\in\mc{W}_{(\Omega,\mc{A},\pi)}$ and $W'\in\mc{W}_{(\Omega',\mc{A}',\pi')}$ be two strong graphons, and let $\alpha$ and $\beta$ be regular anchor sequences in $\Omega$ for $W$ and $\Omega'$ for $W'$, respectively. Suppose that for every partially labeled $\Phi$-decorated multigraph $\mf{F}\in\mc{F}^*$, we have
\[
t_\alpha(\mf{F},W)=t_\beta(\mf{F},W').
\]
Further suppose that for some countable dense subset $\Psi\subset\Phi$ the $p$-norms of $W$ and of its $\alpha_j$-sections satisfy:
\begin{eqnarray}\label{eqn:Carleman}
\sum_{n=1}^\infty \|\tensor*[^\psi]W{}(X,\alpha_j)\|_{2n}^{-1}=\infty & \mbox{ and } & \sum_{n=1}^\infty \|\tensor*[^\psi]W{}\|_{2n}^{-1}=\infty
\end{eqnarray}
for all $j\in\mb{N}^+$ and $\psi\in\Psi$.
Then the anchored graphons $W_\alpha\in\mc{W}_{(\mc{R},\mc{L},\lambda_\alpha)}$ and $W'_\beta\in\mc{W}_{(\mc{R},\mc{L},\lambda'_\beta)}$ satisfy $\lambda_\alpha=\lambda'_\beta$, and $W_\alpha=W'_\beta$ almost everywhere with respect to the common measure.
\end{lemma}
\begin{proof}
Let us first show that $\lambda_\alpha=\lambda'_\beta$. Recall that $\lambda_\alpha$ is the distribution measure of the random variable vector $\left((\tensor*[^\psi]W{}(X,\alpha_1))_{\psi\in\Psi},(\tensor*[^\psi]W{}(X,\alpha_2))_{\psi\in\Psi},\ldots\right)$, where $X\in\Omega$ is a random point with distribution $\pi$, whilst $\lambda_\alpha$ is the distribution measure of the random variable sequence $\left((\tensor*[^\psi]W{}(Y,\alpha_1))_{\psi\in\Psi},(\tensor*[^\psi]W{}(Y,\alpha_2))_{\psi\in\Psi},\ldots\right)$, where $Y\in\Omega'$ is a random point with distribution $\pi'$.

Each of these random variables is real-valued, and we shall first show that their mixed moments are all equal.
Let therefore $(k_{n,\psi})_{n\in\mb{N}^+,\psi\in\Psi}$ be a double-indexed sequence of nonnegative integers, with only finitely many non-zero elements. Let $m\in\mb{N}^+$ be such that $k_{n,\psi}=0$ for all $n>m$, and construct the partially labeled $\Psi$-decorated multigraph $\mf{F}\in\mc{F}^*$ on $m+1$ vertices as follows. First label all but one vertex with the help of the labels $1,\ldots,m$. Then for each $1\leq j\leq m$ and $\psi\in\Psi$ consider the vertex $j$ and the unlabeled vertex, and add $k_{j,\psi}$ edges decorated by $\psi$ between them. Then by construction
\[
E\left(\prod_{n\in\mb{N}^+,\psi\in\Psi}\tensor*[^\psi]W{}(X,\alpha_n)^{k_{n\psi}}\right)=t_\alpha(\mf{F},W).
\]
Similarly
\[
E\left(\prod_{n\in\mb{N}^+,\psi\in\Psi}\tensor*[^\psi]W{}'(Y,\beta_n)^{k_{n,\psi}}\right)=t_\beta(\mf{F},W').
\]

These two are by assumption equal, so all mixed moments are indeed the same. If these were all bounded variables, it would immediately follow that they are equal, as the mixed moments would uniquely determine their joint distribution. In the unbounded case, however, we need an extra property to guarantee uniqueness. By  \cite[Cor. 3a]{KleiberStoyanov},  it is enough that the variables pertaining to $W$ each separately satisfy the Carleman condition, i.e., the family of conditions (\ref{eqn:Carleman}).
Hence $\lambda_\alpha=\lambda'_\beta$.

Now we proceed to show that $W_\alpha(x,y)=W^\prime_\beta(x,y)$ almost everywhere.
To this end we wish to show that the random variables $U_1=(X,Y,W_\alpha(X,Y))$ and $U_2=(X,Y,W'_\beta(X,Y))$ have the same distribution, where $X$ and $Y$ are independent random points in $(\mc{R},\lambda_\alpha)$. 
By definition of the distributions on $(\mc{R},\lambda_\alpha)$ we can generate $X$ and $Y$ by taking the random independent points $X',Y'$ from $(\Omega,\pi)$ and letting $X:=\Gamma_\alpha(X')$ and $Y:=\Gamma_\alpha(Y')$. Since $\alpha$ is regular for $W$, we have with probability 1 that $W_\alpha(X,Y)=W(X',Y')$, and hence
\begin{align*}
U_1=(&(\tensor*[^\psi]W{}(X',\alpha_1))_{\psi\in\Psi},(\tensor*[^\psi]W{}(X',\alpha_2))_{\psi\in\Psi},\ldots,\\
&(\tensor*[^\psi]W{}(Y',\alpha_1))_{\psi\in\Psi},(\tensor[^\psi]W{}(Y'\alpha_2))_{\psi\in\Psi},\ldots,\\
&W(X',Y')).
\end{align*}
For $W'$ we correspondingly take $X''$ and $Y''$ from $(\Omega',\pi')$ instead, and obtain with probability 1 that
\begin{align*}
U_2=(&(\tensor*[^\psi]W{^\prime}(X'',\beta_1))_{\psi\in\Psi},(\tensor*[^\psi]W{^\prime}(X'',\beta_2))_{\psi\in\Psi},\ldots,\\
&(\tensor*[^\psi]W{^\prime}(Y'',\beta_1))_{\psi\in\Psi},(\tensor[^\psi]W{^\prime}(Y''\beta_2))_{\psi\in\Psi},\ldots,\\
&W'(X'',Y'')).
\end{align*}

To compare these two vectors, however, we have to replace their last coordinate with a sequence of real-valued variables with the help of the elements of $\Psi$.

Let therefore
\begin{align*}
\widehat{U}_1:=(&(\tensor*[^\psi]W{}(X',\alpha_1))_{\psi\in\Psi},(\tensor*[^\psi]W{}(X',\alpha_2))_{\psi\in\Psi},\ldots,\\
&(\tensor*[^\psi]W{}(Y',\alpha_1))_{\psi\in\Psi},(\tensor[^\psi]W{}(Y'\alpha_2))_{\psi\in\Psi},\ldots,\\
&(\tensor*[^\psi]W{}(X',Y'))_{\psi\in\Psi})
\end{align*}
and
\begin{align*}
\widehat{U}_2:=(&(\tensor*[^\psi]W{^\prime}(X'',\beta_1))_{\psi\in\Psi},(\tensor*[^\psi]W{^\prime}(X'',\beta_2))_{\psi\in\Psi},\ldots,\\
&(\tensor*[^\psi]W{^\prime}(Y'',\beta_1))_{\psi\in\Psi},(\tensor[^\psi]W{^\prime}(Y''\beta_2))_{\psi\in\Psi},\ldots,\\
&(\tensor*[^\psi]W{^\prime}(X'',Y''))_{\psi\in\Psi}).
\end{align*}

Each mixed moment is determined by nonnegative integers $(a_{n,\psi})_{n\in\mb{N}^+,\psi\in\Psi}$, $(b_{n,\psi})_{n\in\mb{N}^+,\psi\in\Psi}$ and $(c_\psi)_{\psi\in\Psi}$, such that only a finite number of them is non-zero. Assume for instance that $a_{n,\psi}=b_{n,\psi}=0$ for all $n>m$ for some appropriate integer $m$. Let us define a partially labeled $\Psi$-decorated multigraph $\mf{F}\in\mc{F}^*$ on $m+2$ vertices as follows. First label all but two vertices with the labels $1,\ldots,m$, and denote the remaining two by $v_x$ and $v_y$. Then for each $\psi\in\Psi$ add $c_\psi$ parallel edges between $v_x$ and $v_y$, decorating them with $\psi$. Finally for each $1\leq n\leq m$ and $\psi\in\Psi$ add $a_{n,\psi}$ parallel edges decorated with $\psi$ between $v_x$ and the vertex with label $n$, and $b_{n,\psi}$ parallel edges decorated with $\psi$ between $v_y$ and the vertex with label $n$. Then we have
\begin{align*}
E&\left(
\prod_{n\in\mb{N}^+,\psi\in\Psi}\tensor*[^\psi]W{}(X',\alpha_n)^{a_{n,\psi}}\cdot
\prod_{n\in\mb{N}^+,\psi\in\Psi}\tensor*[^\psi]W{}(Y',\alpha_n)^{b_{n,\psi}}\cdot
\prod_{\psi\in\Psi}\tensor*[^\psi]W{}(X',Y')^{c_\psi}
\right)\\
&=t_\alpha(\mf{F},W),
\end{align*}
and \begin{align*}
E&\left(
\prod_{n\in\mb{N}^+,\psi\in\Psi}\tensor*[^\psi]W{^\prime}(X'',\beta_n)^{a_{n,\psi}}\cdot
\prod_{n\in\mb{N}^+,\psi\in\Psi}\tensor*[^\psi]W{^\prime}(Y'',\beta_n)^{b_{n,\psi}}\cdot
\prod_{\psi\in\Psi}\tensor*[^\psi]W{^\prime}(X'',Y'')^{c_\psi}
\right)\\
&=t_\beta(\mf{F},W'),
\end{align*}
which are by assumption equal, and using the conditions in (\ref{eqn:Carleman}), we are done. 
\end{proof}

\begin{remark}\label{re:Carleman}
Note that in condition (\ref{eqn:Carleman}), the second part implies the first for almost all sequences $\alpha$.
Indeed, for $0<c$ let $S_c\subset\Omega$ be the set of points $y\in\Omega$ such that
\[
\sum_{n=1}^\infty \|\tensor*[^\psi]W{}(\cdot,y)\|_{2n}^{-1}\leq c.
\]
Suppose that $\pi(S_c)>0$. We then for $n\geq1$ have by the H\"older inequality that
\begin{eqnarray*}
\|\tensor*[^\psi]W{}\|_{2n}&\geq&\left(\int_{S_c} \left(\int_\Omega \left|\tensor*[^\psi]W{}(x,y)\right|^{2n} \,d\pi(x)\right)\, d\pi(y)\right)^{1/2n}\\
&\geq&\frac{1}{\pi(S_c)^{1-1/2n}}\int_{S_c}  \left(\int_\Omega \left|\tensor*[^\psi]W{}(x,y)\right|^{2n} \,d\pi(x)\right)^{1/2n}\, d\pi(y)\\
&\geq&\frac{1}{\pi(S_c)^{1/2}}\int_{S_c}  \left(\int_\Omega \left|\tensor*[^\psi]W{}(x,y)\right|^{2n} \,d\pi(x)\right)^{1/2n}\, d\pi(y)\\
&=&\frac{1}{\pi(S_c)^{1/2}}\int_{S_c}   \left\|\tensor*[^\psi]W{}(\cdot,y)\right\|_{2n}\, d\pi(y).
\end{eqnarray*}
Thus by convexity
\begin{eqnarray*}
\frac{1}{\|\tensor*[^\psi]W{}\|_{2n}}&\leq&\frac{\sqrt{\pi(S_c)}}{\int_{S_c}   \left\|\tensor*[^\psi]W{}(\cdot,y)\right\|_{2n}\, d\pi(y)}\\
&\leq&\sqrt{\pi(S_c)}\int_{S_c}  \frac{1} {\left\|\tensor*[^\psi]W{}(\cdot,y)\right\|_{2n}}\, d\pi(y),
\end{eqnarray*}
 and so for all $N\in\mb{N}$
\begin{eqnarray*}
\sum_{n=1}^N \frac{1}{\|\tensor*[^\psi]W{}\|_{2n}}&\leq&
\sqrt{\pi(S_c)}\,\sum_{n=1}^N\int_{S_c}  \frac{1} {\left\|\tensor*[^\psi]W{}(\cdot,y)\right\|_{2n}}\, d\pi(y)\\
&=&
\sqrt{\pi(S_c)}\,\int_{S_c} \sum_{n=1}^N \frac{1} {\left\|\tensor*[^\psi]W{}(\cdot,y)\right\|_{2n}}\, d\pi(y)\\
&\leq&
c\sqrt{\pi(S_c)}^3,
\end{eqnarray*}
leading to a contradiction. But if $\pi(S_c)=0$ for all $c=0$, then the sum
\[
\sum_{n=1}^\infty \|\tensor*[^\psi]W{}(\cdot,y)\|_{2n}^{-1}
\]
is infinite for almost all $y\in\Omega$, and therefore also for almost every infinite sequence $\alpha$.
\end{remark}

Our next lemma shows that under a Carleman-type set of conditions, equality of multigraph homomorphism densities, the random canonical forms can be coupled in such a way as to have the corresponding marginals all equal.\\
The reason for us not wanting to have edges between labeled vertices in the test-graphs involved in the coupling is that their absence significantly simplifies and improves the upper bounds on the mixed moments, and thus weaker Carleman-type conditions will suffice. Luckily this is not a restriction here, as the multigraph constructions arising from the mixed moments preserve this property.

\begin{lemma}\label{le:coupling}
Let $W\in\mc{W}_{(\Omega,\mc{A},\pi)}$ and $W'\in\mc{W}_{(\Omega',\mc{A}',\pi')}$ be two strong Lebesguian graphons such that
\[
t(\mf{F},W)=t(\mf{F},W')
\]
for every unlabeled $\Psi$-decorated multigraph $\mf{F}$.
Further suppose that the $p$-norms of $W$
satisfy the Carleman type conditions:
\begin{eqnarray*}
\sum_{n=1}^\infty \|W\|_{2nk}^{-k}=\infty
\end{eqnarray*}
for all $k\in\mb{N}^+$.
Then we can couple sequences $\alpha\in\Omega^{\mb{N}^+}$ with sequences $\beta\in{\Omega'}^{\mb{N}^+}$ such that if $(\alpha,\beta)$ is taken from the joint distribution, then with probability 1
\[
t_\alpha(\mf{F},W)=t_\beta(\mf{F},W')
\]
for every partially labeled $\Psi$-decorated multigraph $\mf{F}\in\mc{F}^*$.
\end{lemma}
\begin{proof}
We shall recursively define a coupling of sequences $\gamma\in\Omega^k$ and $\delta\in{\Omega'}^k$ such that almost surely we for all $\mf{F}\in\mc{F}^*_k$ have that $t_{\gamma}(\mf{F},W)=t_{\delta}(\mf{F},W')$. This is trivial to do for $k=0$. Let us now assume we have such a coupling for sequences of length $k$, and let $(\gamma_1,\ldots,\gamma_k)$ and $(\delta_1,\ldots,\delta_k)$ be chosen from this coupled distribution. Let further $X$ be a random point from $(\Omega,\pi)$, and $Y$ a random point from $(\Omega',\pi')$, and define the random variables
\[
C:=(t_{\gamma_1\ldots\gamma_k X}(\mf{F},W))_{\mf{F}\in\mc{F}^*_{k+1}}
\]
and
\[
D:=(t_{\delta_1\ldots\delta_k Y}(\mf{F},W'))_{\mf{F}\in\mc{F}^*_{k+1}}
\]
with values in $\mb{R}^{\mc{F}^*_{k+1}}$. Our aim is to show that they have the same distribution.

To prove that the joint distributions are equal, we first show that their mixed moments coincide, and then prove that the coordinates of $C$ satisfy the Carleman condition.

Let $m\in\mb{N}^+$, $\mf{F}_1,\ldots,\mf{F}_m\in\mc{F}^*_{k+1}$, and let $q_1,\ldots,q_m$ be non-negative integers. Then for $1\leq j\leq m$ define $\mf{F}_j^{q_j}\in\mc{F}^*_{k+1}$ as the $q_j$-fold product of $\mf{F}_j$ with itself. Let further $\mf{F}\in\mc{F}^*_k$ be defined as the product $\mf{F}_1^{q_1}\ldots \mf{F}_m^{q_m}$ with label $k+1$ removed from the corresponding vertex. Then we have that the corresponding moment of $C$ satisfies
\[
E\left(\prod_{j=1}^m t_{\gamma_1\ldots\gamma_k X}(\mf{F}_j,W)^{q_j}\right)=E\left(t_{\gamma_1\ldots\gamma_k X} (\mf{F}_1^{q_1}\ldots \mf{F}_m^{q_m},W)\right)=t_{\gamma_1\ldots\gamma_k}(\mf{F},W).
\]
Similar arguments yield that the corresponding mixed moment of $D$ is $t_{\delta_1\ldots\delta_k}(\mf{F},W')$, which by hypothesis is the same.

We now have to show that each coordinate of $C$ satisfies the Carleman condition. Let $\mc{F}^*_{k+1}\ni\mf{F}=(F,f)$ and consider
$
C_\mf{F}(X):=t_{\gamma_1\ldots\gamma_k X}(\mf{F},W).
$
Since $F$ does not have any edges between labeled vertices, we have
\begin{align*}
C_\mf{F}(X)&=&\int &\prod_{a=k+2}^{|V(F)|} \left(\prod_{i=1}^{\mr{mult}(v_{k+1}v_{a})} \tensor*[^{f_{(v_{k+1}v_{a})_i}}]W{}(X,x_a) \right)
\prod_{a,b=k+2}^{|V(F)|} \left(\prod_{i=1}^{\mr{mult}(v_{a}v_{b})} \tensor*[^{f_{(v_{a}v_{b})_i}}]W{}(x_a,x_b) \right)\\
&&&\prod_{\substack{1\leq j\leq k\\k+2\leq a\leq |V(F)|}} \left(\prod_{i=1}^{\mr{mult}(v_{j}v_{a})} \tensor*[^{f_{(v_{j}v_{a})_i}}]W{}(\gamma_j,x_a) \right)\,dx_{k+2}\ldots dx_{|V(F)|}
\end{align*}
For almost all anchor sequences $\gamma$ we have that $\int_{\Omega} W(\gamma_j,x) dx$ is finite for all $j$. By the H\"older inequality (using that $\pi(\Omega)=1$) we thus have
\[
|C_\mf{F}(X)|\leq  \ms{C}\prod_{a=k+2}^{|V(F)|}\prod_{i=1}^{\mr{mult}(v_{k+1}v_{a})}  \left(\int_\Omega\left\|W{}(X,x_a)\right\|^{|E(F)|}\,dx_a\right)^{\frac{1}{|E(F)|}},
\]
where $\ms{C}$ is a finite constant depending on $W$, $\mf{F}$ and $\gamma_1,\ldots,\gamma_k$. Consequently, by further applications of H\"older's inequality we have
\begin{eqnarray*}
\|C_\mf{F}\|_n^n&\leq& \ms{C}^n\int_\Omega \prod_{a=k+2}^{|V(F)|}\prod_{i=1}^{\mr{mult}(v_{k+1}v_{a})}  \left(\int_\Omega \left\|W(X,x_a) \right\|^{|E(F)|}\,dx_a\right)^{\frac{n}{|E(F)|}} dX\\
&\leq&\ms{C}^n\prod_{a=k+2}^{|V(F)|}\prod_{i=1}^{\mr{mult}(v_{k+1}v_{a})} \left( \int_\Omega\left(\int_\Omega \left\|W(X,x_a) \right\|^{|E(F)|}\,dx_a\right)^{\frac{n|E(F)|}{|E(F)|}}\,dX\right)^{\frac{1}{|E(F)|}}\\
&=&\ms{C}^n\prod_{a=k+2}^{|V(F)|}\prod_{i=1}^{\mr{mult}(v_{k+1}v_{a})} \left( \int_\Omega\left(\int_\Omega \left\|W(X,x_a) \right\|^{|E(F)|}\,dx_a\right)^{n}\,dX\right)^{\frac{1}{|E(F)|}}\\
&\leq&\ms{C}^n\prod_{a=k+2}^{|V(F)|}\prod_{i=1}^{\mr{mult}(v_{k+1}v_{a})} \left( \int_\Omega\int_\Omega \left\|W(X,x_a) \right\|^{n|E(F)|}\,dx_a\,dX\right)^{\frac{1}{|E(F)|}}\\
&=&\ms{C}^n\prod_{a=k+2}^{|V(F)|}\prod_{i=1}^{\mr{mult}(v_{k+1}v_{a})} \left\|W\right\|_{n|E(F)|}^{n}
\leq\ms{C}^n \left\|W\right\|_{n|E(F)|}^{n|E(F)|}.
\end{eqnarray*}
Therefore with $K:=|E(F)|$ we have
\[
\sum_{n=1}^\infty\|C_\mf{F}\|_{2n}^{-1}\geq \frac{1}{\ms{C}} \sum_{n=1}^\infty\|W\|_{2nK}^{-K}=\infty,
\]
and so the Carleman condition is indeed satisfied.

Thus $C$ and $D$ have the same distribution, and by \cite[Lemma 6.2]{BCL} we can then couple $X$ and $Y$ so that with probability $1$ we have $C=D$. Thus there exist random variables $X'$ from $(\Omega,\pi)$ and $Y'$ from $(\Omega',\pi')$ such that the joint distribution of $(X',Y')\in(\Omega,\Omega')$ satisfies with probability $1$ that
\[
t_{\gamma_1\ldots\gamma_k X'}(\mf{F},W)=t_{\delta_1\ldots\delta_k Y'}(\mf{F},W')
\]
for every $\mf{F}\in\mc{F}_{k+1}$. This extends our coupling to one between $\Omega^{k+1}$ and ${\Omega'}^{k+1}$. Iterating, we obtain the desired coupling between $\Omega^{\mb{N}^+}$ and ${\Omega'}^{\mb{N}^+}$.
\end{proof}

Our final lemma is to show that graphons having equal simple graph densities also have equal multigraph densities, and hence no generality was lost in the assumptions of the previous results. Note that labeled graphs play in the below proof a different role than above, and as such it is not an issue that we here allow (multiple) edges between labeled vertices.

\begin{lemma}\label{le:multigraph}
Let $W_1\in\mc{W}_{(\Omega_1,\mc{A}_1,\pi_1)}$ and $W_2\in\mc{W}_{(\Omega_2,\mc{A}_2,\pi_2)}$ be two countably generated graphons, and assume that $t(\mf{F},W_1)=t(\mf{F},W_2)$ for every simple $\Phi$-decorated graph $\mf{F}$. Then $t(\mf{F},W_1)=t(\mf{F},W_2)$ for every $\Phi$-decorated multigraph $\mf{F}=(F,f)$.
\end{lemma}
\begin{proof}
We shall proceed by induction on the number of parallel edges in $F$. The base case is given by the assumption. Let $v_i$ and $v_j$ be two vertices connected by more than one edge, and let $\varphi\in\Phi$ be the decoration on one of them. Denote by $\mf{F}'$ the decorated multigraph obtained by deleting one $\varphi$-decorated edge between $v_i$ and $v_j$. Let further $\mf{F}^k$ denote the decorated multigraph obtained by adding a path of length $k$ between $v_i$ and $v_j$ in $\mf{F}'$, decorating each edge in the path with $\varphi$. Thus $\mf{F}^1=\mf{F}$, but for each $k>1$ the multigraph $\mf{F}^k$ has fewer parallel edges than $\mf{F}$. Hence by the inductive assumption we have that $t(\mf{F}^k,W_1)=t(\mf{F}^k,W_2)$ for each $k>1$. Let us now label all the multigraphs $\mf{F}^k$ and $\mf{F}'$ such that $v_i$ receives the label $1$ whilst $v_j$ receives the label $2$. Then $\mf{F}^k$ is the product of $\mf{F}'$ with the path $\mf{P}_{k+1}$ of length $(k+1)$ with its two endpoints labeled $1$ and $2$ respectively and each edge decorated with $\ph$, and we may write
\[
t(\mf{F},W_1)=\int_{\Omega_1^2} W(x,y)t_{xy}(\mf{F}',W_1)d\pi_1(x)d\pi_1(y),
\]
and
\[
t(\mf{F}^k,W_1)=\int_{\Omega_1^2} t_{xy}(\mf{P}_{k+1},W)t_{xy}(\mf{F}',W_1)d\pi_1(x)d\pi_1(y).
\]

Note that
\[
t_{xy}(\mf{P}_{k+1},W)=\int_{\Omega_1^{k-1}} \tensor*[^{\varphi}]W{}(x,x_1)\cdots\tensor*[^{\varphi}]W{}(x_{k-1},y) d\pi(x_1)\ldots d\pi(x_{k-1}).
\]

Since $\tensor*[^{\varphi}]W{_1}\in L^2(\Omega_1^2,\mb{R})$, it is a self-adjoint Hilbert-Schmidt integral operator on $L^2(\Omega_1,\mb{C})$, and thus has a spectral decomposition

\begin{equation*}
\tensor*[^{\varphi}]W{_1}(x,y)=\sum_{n,m=1}^{\infty} \lambda_{n,m} \zeta_n(x)\zeta_m(y)
\end{equation*}
in the $L^2$ sense, where $(\zeta_n)$ is an orthonormal system in $L^2(\Omega_1,\mb{C})$.

Then we obtain by induction on $k$ that for every $k>1$,
\[
t_{xy}(\mf{P}_{k+1},W_1)=\sum_{n,m=0}^\infty \lambda_{n,m}^k\zeta_n(x)\zeta_m(y)
\]
in $L^2$, whereby
\[
t(\mf{F}^k,W_1)=\sum_{n,m=1}^\infty \lambda_n^k\int_{\Omega^2}\zeta_n(x)\zeta_m(y)t_{xy}(\mf{F}',W_1) d\pi_1(x)d\pi_1(y),
\]
since $t_{xy}(\mf{F}',W_1)$ lies in $L^2$ by the assumption on $W_1$.

Similarly, with the spectral decomposition
\[
\tensor*[^{\varphi}]W{_2}(x,y)=\sum_{n,m=1}^{\infty} \mu_{n,m} \eta_n(x)\eta_m(y),
\]
we obtain
\begin{equation}\label{eqn:spectral}
0=t(\mf{F}^k,W_1)-t(\mf{F}^k,W_2)=\sum_{n,m=1}^\infty a_{n,m}\lambda_{n,m}^k-b_{n,m}\mu_{n,m}^k
\end{equation}
for every $k\geq 2$, where the parameters
\[
a_{n,m}:=\int_{\Omega_1^2} \zeta_n(x)\zeta_m(y)t_{xy}(\mf{F}',W_1) d\pi_1(x)d\pi_1(y)
\]
and
\[
b_{n,m}:=\int_{\Omega_2^2} \eta_n(x)\eta_m(y)t_{xy}(\mf{F}',W_2) d\pi_2(x)d\pi_2(y)
\]
are independent of $k$.

Since each non-zero eigenvalue has finite multiplicity and the only possible accumulation point is $0$, the asymptotic behaviour of the right hand side in (\ref{eqn:spectral}) dictates that all the terms have to cancel, i.e., for each $c\in\mb{R}\backslash\{0\}$ we have that
\[
\sum_{\lambda_{n,m}=c} a_{n,m}=\sum_{\mu_{n,m}=c} b_{n,m}.
\]

Then
\[
t(\mf{F},W_1)=\sum_{n,m=1}^\infty \lambda_{n,m}\int_{\Omega^2}\zeta_n(x)\zeta_m(y)t_{xy}(\mf{F}',W_1) d\pi_1(x)d\pi_1(y)=\sum_{n,m=1}^\infty a_{n,m}\lambda_{n,m},
\]
similarly
\[
t(\mf{F},W_2)=\sum_{n,m=1}^\infty b_{n,m}\mu_{n,m},
\]
and the claim follows.
\end{proof}

\begin{remark}
By Lemma \ref{le:countable}, countable generation is not actually needed.
\end{remark}

\section{Proof of main theorem}

Having extended the intermediate steps of \cite{BCL} to the significantly more general context of our investigations, we are now ready to bring together the elements of the previous sections to prove our main result.

\begin{proof}[\textbf{Proof of Theorem \ref{thm:main}}]

Part (i): By Lemma \ref{le:multigraph}, equality of simple graph densities implies equality of multigraph densities, and thus we may apply Lemma \ref{le:coupling} to our two graphons.
If we choose the anchor sequences $\alpha\in\Omega^{\mb{N}^+}$ and $\beta\in{\Omega'}^{\mb{N}^+}$ from the joint distribution given by Lemma \ref{le:coupling}, almost all such choices satisfy
\[
t_\alpha(\mf{F},W)=t_\beta(\mf{F},W')
\]
for every partially labeled $\Psi$-decorated multigraph $\mf{F}\in\mc{F}^*$, and but for a further null-set they yield anchor sequences that are regular (Lemma \ref{le:regular}). Hence, taking into consideration Remark \ref{re:Carleman}, we can choose sequences $\alpha$ and $\beta$ that satisfy the conditions of Lemma \ref{le:equality}.
Consequently the anchored graphons $W_\alpha$ and $W_\beta'$ are isomorphic mod $0$ through the identity map. 
If we now could show that $\Gamma_\alpha$ is an isomorphism mod $0$ between $W$ and $W_\alpha$, and that similarly $\Gamma_\beta$ is an isomorphism mod $0$ between $W'$ and $W'_\beta$, our proof would be complete. Due to symmetry we shall only show the first isomorphism.

First note that since $(\Omega,\mc{A},\pi)$ is a Lebesgue space, the mapping $\Gamma_\alpha$ is not only measurable and measure preserving as a mapping $(\Omega,\mc{A},\pi)\to (\mc{R},\mc{K},\kappa_\alpha)$, but also as a mapping $(\Omega,\mc{A},\pi)\to (\mc{R},\mc{L},\lambda_\alpha)$.
Let 
\[
S:=\{x\in\Omega\colon W_\alpha(\Gamma_\alpha(x),\Gamma_\alpha(y))=W(x,y) \mbox{ for almost all }y\}.
\]
Since $\alpha$ is a regular anchor sequence, we have $\pi(S)=1$. Also, because $W$ is almost twin-free, we can find a null-set $N\subset\Omega$ such that each twin-class of $W$ has at most one point in $T:=S\backslash N$. Let $\Gamma_\alpha'$ be the restriction of $\Gamma_\alpha$ to $T$. Then it can easily be seen that $T$ is injective. By \cite[Section 2.5]{Rohlin}, injective measure preserving maps between Lebesgue spaces have an almost everywhere defined measurable inverse. Thus $\Gamma_\alpha'$ is an isomorphism mod $0$, and then so is $\Gamma_\alpha$.

Part (ii):
First, by Corollary \ref{cor:Lebesgue} we can find two almost twin-free strong Lebesguian graphons $U\in\mc{W}_{(O,\mc{B},\rho)}$ and $U'\in\mc{W}_{(O',\mc{B}',\rho')}$ and corresponding weak isomorphisms $\gamma$ and $\gamma'$ from $W$ and $W'$ to $U$ and $U'$, respectively. By part (i) we then have that $U$ and $U'$ are isomorphic mod $0$, hence for some measure preserving map $\eta:O\to O'$ we have $U=(U')^\eta$. Now let $\mu:\Omega\to O'$ be defined through $\mu(x):=\eta(\gamma(x))$. Then clearly $W=(U')^\mu$ almost everywhere, and the maps $\mu$ and $\gamma'$ are measure preserving from the completions of $W$ and $W'$, respectively, into $U'$.

\end{proof}

\section{A moment-indeterminate graphon}\label{Section:ex}

In this last section we wish to provide an example of two real-valued graphons that possess the same homomorphism densities, but are not weakly isomorphic. In fact, in our example they are not inducing the same probability measure on $\mb{R}$, so they are even distinguishable when forgetting about the geometry coming from the underlying product space.

Let $\sigma$ and $\tau$ be two probability distributions on $\mb{N}$ with
finite moments and having the same moments (such distributions exist,
see e.g. \cite[Cor. 3.4]{Ped}). Denote their n-th moments by $M_n$ $(n\geq0)$.\\
Let further $\{S_i\}_{i\in\mb{N}}$ and $\{T_j\}_{j\in\mb{N}}$ be two partitions of $[0,1]$ into measurable sets such that $\lambda(S_i)=\sigma(\{i\})$ and $\lambda(T_j)=\tau(\{j\})$ for all $i,j\in\mb{N}$. Consider the functions $f_\sigma, f_\tau: [0,1]\rightarrow\mb{R}$ defined by
\begin{eqnarray*}
f_\sigma(x):=n_x & \mbox{ whenever } & x\in S_{n_x},\\
f_\tau(x):=m_x & \mbox{ whenever } & x\in T_{m_x},
\end{eqnarray*}
respectively, and let $W_\sigma, W_\tau: [0,1]^2\rightarrow\mb{R}$ be defined by $W_\sigma(x,y):=f_\sigma(x)f_\sigma(y)$ and $W_\tau(x,y)=f_\tau(x)f_\tau(y)$, respectively.\\
Let $\mf{F}$ be an $\mb{R}$-decorated graph with each edge decorated with $1$. By linearity, it is enough to show that $t(\mf{F},W_\sigma)=t(\mf{F},W_\tau)$ for every such $\mf{F}$. Let the elements of $V(F)$ be denoted by $v_1,v_2,\ldots,v_k$, and let $d_i$ denote the degree of vertex $v_i$. It can then easily be seen that we have
\begin{eqnarray*}
t(\mf{F},W_\sigma)&=&\int\limits_{x_1,\ldots,x_k\in[0,1]} \prod_{v_iv_j\in E(F)} W(x_i,x_j) dx_1\ldots x_k\\
&=&\int\limits_{x_1,\ldots,x_k\in[0,1]} \prod_{v_iv_j\in E(F)} f_\sigma(x_i)f_\sigma(x_j) dx_1\ldots x_k\\
&=&\prod_{i=1}^k \int_{[0,1]} f_\sigma(x_i)^{d_i} dx_i=\prod_{i=1}^k M_{d_i}.
\end{eqnarray*}
Similar calculations yield $t(\mf{F},W_\sigma)=\prod_{i=1}^k M_{d_i}$, and so the two graphons do indeed have the exact same generalized moments.

It remains to be shown that $W_\sigma$ and $W_\tau$ yield different probability measures on $\mb{R}$, but this easily follows from their product form, and the fact that $\sigma\neq\tau$.

\bibliographystyle{amsplain}

\end{document}